\documentclass[11pt]{amsart}

\usepackage{amsmath,amsthm,amssymb,url,amssymb}

\usepackage[utf8]{inputenc}
\usepackage[T1]{fontenc}
\usepackage{lmodern}
\newcommand{\mA}{\mathcal{A}}
\newcommand{\mB}{\mathcal{B}}

\newcommand{\mE}{\mathcal{E}}

\newcommand{\mI}{\mathcal{I}}

\newcommand{\mM}{\mathcal{M}}

\newcommand{\mP}{\mathcal{P}}

\newcommand{\mR}{\mathcal{R}}
\newcommand{\mS}{\mathcal{S}}
\newcommand{\mT}{\mathcal{T}}

\newcommand{\fm}{\mathfrak{m}}

\newcommand{\bfC}{\mathbf{C}}

\newcommand{\bfF}{\mathbf{F}}
\newcommand{\bfG}{\mathbf{G}}

\newcommand{\bfQ}{\mathbf{Q}}

\newcommand{\bfT}{\mathbf{T}}

\newcommand{\bfZ}{\mathbf{Z}}

\newcommand{\Oo}{\mathcal{O}}

\newcommand{\AL}{\mathbf{A}_L}

\newcommand{\ov}{\overline}

\newcommand{\be}{\begin{equation}}
\newcommand{\ee}{\end{equation}}
\newcommand{\bes}{\begin{equation*}}
\newcommand{\ees}{\end{equation*}}

\newcommand{\bs}{\begin{split}}
\newcommand{\es}{\end{split}}
\newcommand{\bss}{\begin{split*}}
\newcommand{\ess}{\end{split*}}

\newcommand{\bmat}{\left[ \begin{matrix}}
\newcommand{\emat}{\end{matrix} \right]}
\newcommand{\bsmat}{\left[ \begin{smallmatrix}}
\newcommand{\esmat}{\end{smallmatrix} \right]}

\newcommand{\bml}{\begin{multline}}
\newcommand{\eml}{\end{multline}}
\newcommand{\bmls}{\begin{multline*}}
\newcommand{\emls}{\end{multline*}}

\DeclareMathOperator{\diag}{diag}

\DeclareMathOperator{\End}{End}
\DeclareMathOperator{\Ext}{Ext}

\DeclareMathOperator{\Frob}{Frob}
\DeclareMathOperator{\Gal}{Gal}
\DeclareMathOperator{\GL}{GL}
\DeclareMathOperator{\GSp}{GSp}

\DeclareMathOperator{\Hom}{Hom}

\DeclareMathOperator{\Sel}{Sel}
\DeclareMathOperator{\SL}{SL}

\DeclareMathOperator{\Sp}{Sp}

\DeclareMathOperator{\val}{val}

\newcommand{\tr}{\textup{tr}\hspace{2pt}}

\usepackage[all]{xy}
\usepackage{amsthm}
\usepackage{amssymb, amsfonts}
\SelectTips{cm}{10}\UseTips
\theoremstyle{plain}
\newtheorem{thm}{Theorem}
\newtheorem{prop}[thm]{Proposition}
\newtheorem{cor}[thm]{Corollary}
\newtheorem{lemma}[thm]{Lemma}
\newtheorem{conj}[thm]{Conjecture}

\theoremstyle{definition}
\newtheorem{definition}[thm]{Definition}

\newtheorem{rem}[thm]{Remark} 
\newtheorem{assumption}[thm]{Assumption}

\numberwithin{thm}{section}
\numberwithin{equation}{section}

\usepackage{multibib}

\newcites{app}{Appendix References}

\usepackage{booktabs}
\usepackage{youngtab}
\usepackage{longtable}

\DeclareMathOperator{\SK}{SK}

    \DeclareFontFamily{U}{wncy}{}
    \DeclareFontShape{U}{wncy}{m}{n}{<->wncyr10}{}
    \DeclareSymbolFont{mcy}{U}{wncy}{m}{n}
    \DeclareMathSymbol{\Sh}{\mathord}{mcy}{"58}

\newcommand\CFsnoweight{{\mathcal S}} 
\newcommand\CFswtgp[2]{\CFsnoweight_{#1}(#2)}
\newcommand\paramodulargroup[1]{{\rm K}(#1)}
\newcommand\KN{\paramodulargroup N}
\newcommand\SkKN{\CFswtgp k\KN}
\newcommand\StwoKN{\CFswtgp2\KN}
\newcommand\StwoKlevel[1]{\CFswtgp2{\paramodulargroup{#1}}}
\newcommand\StwoKsto{\StwoKlevel{731}}   
\newcommand\SfourKN{\CFswtgp4\KN}
\newcommand\SfourKlevel[1]{\CFswtgp4{\paramodulargroup{#1}}}
\newcommand\SfourKsto{\SfourKlevel{731}}

\newcommand\JkNcusp[2]{{\rm J}_{#1,#2}^{\rm cusp}}
\newcommand\Jkmcusp{\JkNcusp km}
\newcommand\JkjNcusp{\JkNcusp k{jN}}

\def\Grit{\operatorname{Grit}}
\def\JRMJ{{\mathcal J}}
\def\JRMJdeps{\JRMJ_d^\epsilon}
\def\JRMJdpeps{\JRMJ_{d,p}^\epsilon}
\def\JRMJdsigns{\JRMJ_d^{\text{signs}}}
\def\JRMJdpsigns{\JRMJ_{d,p}^{\text{signs}}}

\def\detmax{{\det_{\max}}}
\def\AL{\operatorname{AL}}

\newcommand\lra{\longrightarrow}
\newcommand\mymod{\text{ mod }}

\def\Z{{\mathbb Z}}
\def\F{{\mathbb F}}
\newcommand\Fsubp[1]{\F_{\!#1}}
\def\Fp{\Fsubp p}
\def\Ffive{\Fsubp5}
\def\C{{\mathbb C}}
\def\Q{{\mathbb Q}}

\author[Tobias Berger, Krzysztof Klosin]{Tobias Berger$^1$ \and
Krzysztof Klosin$^2$ \\with an appendix by Cris Poor$^3$, Jerry Shurman$^4$, and David S. Yuen$^5$}
\address{$^1$School of Mathematics and Statistics, University of Sheffield, Hicks Building, Hounsfield Road, Sheffield S3 7RH, UK.}
\email{tberger@cantab.net}
\address{$^2$Department of Mathematics, Queens College of the City University of New York,  65-30 Kissena Blvd, Queens, NY 11367, USA}
\email{krzysztof.klosin@qc.cuny.edu}
\address{$^3$Department of Mathematics, Fordham University, Bronx, NY 10458, USA}
\email{poor@fordham.edu}

\address{$^4$Department of Mathematics, Reed College, Portland, OR 97202 USA}
\email{jerry@reed.edu}

\address{$^5$Department of Mathematics, University of Hawaii, 2565 McCarthy Mall, Honolulu, HI, 96822 USA}
\email{yuen888@hawaii.edu}

\begin{document}
\title[Deformations of Saito-Kurokawa type ]{Deformations of Saito-Kurokawa type and the Paramodular Conjecture}

\begin{abstract} We study short crystalline, minimal, essentially self-dual deformations of a mod $p$ non-semisimple Galois representation 
$\ov{\sigma}$ with $\ov{\sigma}^{\rm ss}=\chi^{k-2} \oplus \rho \oplus \chi^{k-1}$, where $\chi$ is the mod $p$ 
cyclotomic character and $\rho$ is an  absolutely irreducible reduction of the Galois representation $\rho_f$ attached 
to a cusp form $f$ of weight $2k-2$. We show that if the Bloch-Kato Selmer groups $H^1_f(\bfQ, \rho_f(1-k)\otimes \bfQ_p/\bfZ_p)$ 
and $H^1_f(\bfQ, \rho(2-k))$ have order $p$, and there exists a characteristic zero absolutely irreducible 
deformation of $\ov{\sigma}$ then the universal deformation ring is a dvr. When $k=2$ this allows us to establish the modularity part of the Paramodular Conjecture 
in cases when one can find a suitable congruence of Siegel modular forms. As an example we prove the modularity of  an abelian 
surface of conductor 731. When $k>2$, we obtain an $R^{\rm red}=T$ theorem showing modularity of all such deformations of $\ov{\sigma}$.\end{abstract}

\thanks{The first author's research was supported by an LMS Research in Pairs Grant and the EPSRC Grant EP/R006563/1. The second author was supported by a Collaboration for Mathematicians Grant \#578231  from the Simons Foundation
 and by a PSC-CUNY award jointly funded by the Professional Staff Congress and the City
University of New York.}

\maketitle

\section{Introduction} 
In analogy with the Taniyama-Shimura Conjecture (proved by Wiles et al.), in the 1980s Yoshida proposed a conjecture postulating that abelian surfaces over $\bfQ$ should correspond to Siegel modular forms. The most important progress on this conjecture is due to Tilouine and Pilloni (\cite{Tilouine06, Pilloni12}).  Their results show $p$-adic modularity of abelian surfaces $A$ (see also forthcoming work of Boxer, Calegari, Gee, and Pilloni on modularity, see \cite{BCGP})  under the assumption that the residual Galois representation  $A(\ov{\bfQ})[p]$ has image containing ${\rm PSp}_4(\bfF_p)$ and, more  significantly, is itself modular. The latter assumption essentially calls for an analogue of Serre's conjecture for $\ov{\sigma}_A$, which at present appears out of reach. 

Yoshida himself proved some cases of his conjecture in \cite{Yoshida80} when $\End_{\bfQ}(A)\neq\bfZ$.
In 2014, Brumer and Kramer proposed a refinement of  Yoshida's Conjecture by specifying the level of the Siegel modular form \cite{BrumerKramer14}. This is sometimes referred to as the \emph{Paramodular Conjecture}. 
\begin{conj}[Brumer - Kramer]  \label{conj:paramodular1}  For every isogeny class of abelian surfaces $A_{/\bfQ}$ of conductor $N$ with $\End_{\bfQ} A = \bfZ$ there exists a  weight 2 Siegel modular form $F$,  which is not in the space spanned by the Saito-Kurokawa lifts, has level $K(N)$ and rational eigenvalues,  where $K(N)$ is the paramodular group of level $N$ defined by $K(N) = \gamma M_{4}(\bfZ) \gamma^{-1} \cap \Sp_4(\bfQ)$ with $\gamma = \diag[1,1,1,N]$.  The $L$-series of $A$ and $F$ should agree and the $p$-adic representation of $T_{p}(A) \otimes \bfQ_{p}$ should be isomorphic to the one associated to $F$ for any $p$ prime to $N$ where $T_{p}(A)$ is the $p$-adic Tate module.
\end{conj}

Brumer-Kramer \cite{BrumerKramer14} and Poor-Yuen \cite{PoorYuen15} verified Conjecture \ref{conj:paramodular1} in several cases by proving the non-existence of an abelian surface for certain (small) conductors $N$ for which there are no non-lift Siegel modular forms of level $K(N)$. 
When $A$ acquires extra endomorphisms over a quadratic field the paramodularity of $A$ has in some cases been proven
 using theta lifts, see \cite{JLR12} and \cite{BDPS}.

At the time of writing this article the only genuinely symplectic case (i.e., concerning surfaces whose associated Galois representations are absolutely irreducible and not induced from $2$-dimensional representations)  of Conjecture \ref{conj:paramodular1} in which the abelian surface exists that had  been fully verified was the case of the abelian surface of conductor $277$, see \cite{Brumeretal19preprint}, which now also contains a proof for abelian surfaces of conductors 353 and 587.

In this article we study Galois representations arising from abelian surfaces (as in Conjecture \ref{conj:paramodular1}) with rational torsion. We propose a method to prove modularity of such representations which we carry out here under certain assumptions.  This case is not covered by the work of Tilouine and Pilloni et al. It has the advantage of not requiring modularity of  $\ov{\sigma}_A$  and replacing this assumption with a construction of congruences between paramodular Saito-Kurokawa lifts and stable Siegel modular forms.
 In other words we do not have a need for an analogue of Serre's Conjecture.  While potentially more accessible, proving the existence of such congruences is still a considerable problem. In this article however we mostly focus on the Galois representation side (deformation theory) of the modularity problem, and make use of results where the required congruence has been constructed by others  (mostly by Poor and Yuen). 
For the purposes of this article Poor, Yuen and Shurman have kindly agreed to prove the existence of such a congruence in the case of forms of paramodular level 731 (see Appendix). See also \cite{PoorYuen15} for other examples. 
A full modularity result would require the existence of such congruences in a  more general context, a problem which is a subject of joint work in progress with J. Brown.  As a consequence of the results proved in this article we are able to establish the modularity part of the Paramodular Conjecture in new (genuinely symplectic) cases (in particular the first composite level case) and provide a way to verify more, when the appropriate ingredients are supplied.

Let us explain the results and the methods of this paper in more detail. Let $p>3$ be a prime such that $p \nmid N$. Let $A$ be as in Conjecture \ref{conj:paramodular1} and suppose that $A$ has a rational $p$-torsion point,  semi-abelian reduction at $\ell \mid N$, and a polarization of degree prime to $p$.  This implies that the semisimplification $\ov{\sigma}_A^{\rm ss}$ of the residual Galois representation $\ov{\sigma}_A: G_{\bfQ} \to \GL_4(\bfF_p)$ attached to $A$ (i.e., afforded by the $p$-torsion Galois module $A(\ov{\bfQ})[p]$) has the form $$\ov{\sigma}_A^{\rm ss} \cong 1 \oplus \chi \oplus \rho^{\rm ss},$$ where $\chi$ is the mod $p$  cyclotomic character and $\rho $ is a two-dimensional representation.  Assume that $\rho$ is absolutely irreducible (it is automatically odd). Serre's Conjecture 
(now a theorem of Khare and Wintenberger) implies that $\rho$ arises as a mod $p$ reduction of the Galois representation $\rho_f$ attached to a modular form $f$ of weight 2 and level $\Gamma_0(N)$. If the sign of the functional equation for $f$ is $-1$, then  $\ov{\sigma}_A^{\rm ss}$ is the mod $p$ reduction of the  Galois representation attached to a weight 2, level $K(N)$ Siegel Hecke eigenform $\SK(f)$ which is the \emph{paramodular Saito-Kurokawa lift} of $f$. In fact, regardless of the sign of $f$, $\ov{\sigma}^{\rm ss}$ arises from a congruence level  Saito-Kurokawa lift of $f$,  see section \ref{s9.2} for  examples.

 Our first result shows that one can construct a Galois-invariant lattice $L$ in the space $\bfQ_p^4$ with $G_{\bfQ}$-action via $\sigma_A$ with respect to which the residual Galois representation is non-semisimple. More precisely, we can ensure that it is block-upper-triangular with a specific order of the Jordan-Hölder factors on the diagonal ($1$, $\rho$, $\chi$) - this order plays an important role in controlling the deformations (see below).   For the modularity argument it is crucial to know that there is only one isomorphism class of such residual representations as long as we require that they be short crystalline, minimally ramified and give rise to desired extensions of the Jordan-Hölder factors of $\ov{\sigma}_A$. The existence of such a lattice can be proved by essentially following the standard method of Ribet (cf. Proposition 2.1 in \cite{Ribet76}) adapted to a higher-dimensional setting. The uniqueness of the residual representation is trickier and requires us to study iterated Fontaine-Laffaille extensions as well as to control ramification at the primes dividing $N$.

We then study short crystalline, minimal and essentially self-dual deformations of the residual representation $\ov{\sigma}_A$. Let $R$ denote the reduced universal deformation ring. Its structure relevant for our purposes can be controlled by two sets of data: its ideals of reducibility $\{I_{\mP}\}_{\mP}$ corresponding to all possible partitions $\mP$ of the set of Jordan-Hölder factors of $\ov{\sigma}_A$ and the quotients $\{R/I_{\mP}\}_{\mP}$. Roughly speaking, the former control trace-irreducible deformations, while the latter  the trace-reducible ones. In the case when $I_{\mP}$ is the total ideal of reducibility (i.e., corresponds to the most refined partition), the quotient $R/I_{\mP}$ can be shown to be related to the Bloch-Kato Selmer group $H_{-1}$, where $H_i:=H^1_f(\bfQ, \rho_f(i)\otimes \bfQ_p/\bfZ_p)$,  by a generalization of the approach used in \cite{BergerKlosin13}. On the other hand understanding of  the ideals of reducibility $I_{\mP}$ requires an entirely new approach as the methods used in \cite{BergerKlosin13} cannot be extended to our current situation.

 To overcome this problem we draw on ideas of Bella{\"\i}che and Chenevier contained in sections 8 and 9 of their indispensable book \cite{BellaicheChenevierbook} which has served us as an inspiration on many occasions. There  a problem concerning representations with multiple Jordan-Hölder factors is studied in characteristic zero (as opposed to our characteristic $p$ situation). We, too, are able to prove that in our situation (assuming 
 that the divisible Bloch-Kato Selmer group  $H_0$ is of corank at most 1)  all the (a priori different) ideals of reducibility in fact coincide and as a consequence we are able to deduce that they are principal. This approach poses a considerable amount of technical difficulties because, among other things, we are not able to control ramification in the same way as Bella{\"\i}che and Chenevier do, one reason being that  their proof of the splitting of extensions at ramified primes (see proof of Proposition 8.2.10 in \cite{BellaicheChenevierbook}) does not extend to characteristic $p$.

So, while following the general strategy of \cite{BellaicheChenevierbook} we need a different way of manufacturing the necessary ingredients, and in particular are led to working  with a specific minimality condition. Happily this turns out to also be the correct condition for yielding deformations corresponding to Siegel modular forms of  squarefree paramodular level (see Proposition \ref{prop8.4}). In some sense section \ref{The ideals of reducibility}, where the ideal of reducibility is studied, comprises the technical heart of the paper.

We show that  if $H_0$ is of corank $\leq 1$ and $\# H_{-1}=p$ (that $\#H_{-1}\geq p$ is automatic thanks to the existence of the lattice $L$), then  $R$ is a discrete valuation ring (see Theorem \ref{dvrness}). It means that the deformation $\sigma_A$ of $\ov{\sigma}_A$ is the unique characteristic zero deformation. To conclude modularity of $A$ we now need a non-lift Siegel modular form $F$ of weight 2 and level $K(N)$ whose Galois representation $\sigma_F$ reduces (after semisimplification) to $\ov{\sigma}_A^{\rm ss}$. This is achieved by exhibiting the existence of a congruence between a candidate $F$ and the paramodular Saito-Kurokawa lift $\SK(f)$ (see section \ref{Examples} and the Appendix).   Irreducibility of $\sigma_F$ (which we prove) and uniqueness of the non-semisimple residual representation (discussed above) now guarantee that $\sigma_F$ is a characteristic zero deformation of $\ov{\sigma}_A$. So, we must have $\sigma_A \cong \sigma_F$, proving the modularity of $A$.  This can be viewed as the main application of our method in this paper, stated as Theorem \ref{paramodularity}. If a matching bound on $\bfT/J$ were available  (see next paragraph) our method would in fact be enough to prove a full $R=T$ theorem. We plan to address this problem in a future paper.

Let us note that we prove our results for a residual representation $\ov{\sigma}$ such that $\ov{\sigma}^{\rm ss} \cong \chi^{k-2} \oplus \chi^{k-1} \oplus \rho$ for any integer $k \geq 2$. The case $k=2$ (discussed above) is the most interesting because it corresponds to abelian surfaces and thus to Conjecture \ref{conj:paramodular1}. 
In fact for even $k>9$ and $N=1$ we prove a full modularity result $R=\bfT$ using a congruence result of Brown \cite{Brown11}, which provides us with the desired matching bound on $\bfT/J$ in this case.  

Let us now comment on the restrictiveness of our results. The most serious assumptions are the ones on the Selmer groups. It is worth noting that while for $H_{-1}$  we require a specific bound on its order, for $H_0$ we only require that is of corank 1. In fact, it is important that for our method the corank assumption of the latter group is sufficient since $H_0$ is infinite 
  if the central $L$-value of $f$ is zero. It is here that the order of the Jordan-Hölder factors on the diagonal of the residual representation is crucial - a different order could swap the conditions on $H_0$ and $H_{-1}$ making our theorem empty  in this case. 

 On the other hand the Selmer group $H_{-1}$ is non-critical which poses certain difficulty in computing its order. We overcome this by using Kato's result \cite{Kato04} towards the Main Conjecture of Iwasawa Theory and slightly adapting a control theorem of \cite{SkinnerUrban14} to relate the order to a special value of a $p$-adic (rather than classical) $L$-function of $f$ (cf. section \ref{noncritical}).

The assumptions on the Selmer groups are central to our method and in fact without them $R$ cannot be a dvr.

We also require that $p$ is not a congruence prime for $f$ and that each $\ell \mid N$ satisfies $p \nmid 1+w_{f, \ell}\ell$, where  $w_{f, \ell}$ is the local (at $\ell$) Atkin-Lehner sign of $f$. The first condition allows us to relate the order of $R/I_{\mP}$ to the order of $H_{-1}$, while the latter allows us to prove that the classes in $H^1(\bfQ, \rho_f\otimes \bfQ_p/\bfZ_p)$ are automatically unramified at $\ell \mid N$  (see Proposition \ref{Selmer groups 2}). It is conceivable that both of these conditions could be relaxed, but we do not know of a way to do it. What is important is that these conditions are often satisfied for small primes  which has been our goal since it is often the case that abelian surfaces possess rational  $p$-torsion for such primes. Unfortunately, the method to prove principality of the ideal of reducibility requires that $p$ does not divide $d!$, where $d$ is the dimension of the Galois representations (cf. section 1.2 of \cite{BellaicheChenevierbook}), which forces us to exclude $p=2,3$.
As for elliptic curves the size of the torsion subgroup of rational
points on abelian surfaces is conjectured to be bounded, but so far this has only been proven in
special cases. Importantly for us it is known though by work e.g. of Flynn \cite{Flynn90} that there are infinitely many
abelian surfaces with rational torsion points for certain primes $p$.

While we  hope that this article is only the first half of a larger undertaking of proving a full $R=\bfT$ theorem for abelian surfaces in Conjecture \ref{conj:paramodular1} with rational torsion (the second half being the mentioned work in progress on constructing congruences), it is worth noting that even without the matching work on the Hecke side, it allows us to prove modularity of several new examples. In section \ref{Examples} we work out the details for an abelian surface of conductor 731, and discuss other examples. 

The paper is organized as follows. We begin by assembling necessary results concerning Selmer groups in section \ref{Setup} and define the relevant deformation problem  in section \ref{Deformations}. The lattice $L$ is constructed in section \ref{Lattice} and the uniqueness of the residual representation is proved in section \ref{Uniqueness of iterated residual extensions}. Section \ref{The ideals of reducibility} is devoted to the study of the ideals of reducibility $I_{\mP}$, while in Section \ref{Cyclicity} we study the quotients $R/I_{\mP}$ and conclude the proof that $R$ is a dvr. In section \ref{Application to the Paramodular Conjecture} we prepare the ground for applications to examples of abelian surfaces satisfying our conditions and state the application to the Paramodular Conjecture (Theorem \ref{paramodularity}). Finally, in section \ref{Examples} we prove the modularity of an abelian surface of conductor $N=731$ and discuss other examples. We conclude the paper by proving an $R=\bfT$ theorem in the case of $k>2$ and $N=1$ in section \ref{Modularity theorem}.

We would like to thank David Savitt for helping us with the proof of Proposition \ref{uni1} for the prime $p$ and Jim Brown, Armand Brumer, Kenneth Kramer, Chris Skinner and Eric Urban for helpful conversations related to the topics of this article.   We would also like to thank the anonymous referee for some helpful comments on the manuscript.

\section{Selmer groups} \label{Setup} 
  For each prime $\ell$ of $\bfQ$ we fix embeddings $\ov{\bfQ} \hookrightarrow \ov{\bfQ}_{\ell} \hookrightarrow \bfC$.  Let $p>3$ be a prime.
Throughout this paper $E$ will denote a sufficiently large finite extension of $\bfQ_p$, $\Oo$ its valuation ring with uniformizer $\varpi$ and residue field $\bfF$. Let $\epsilon$ be the $p$-adic cyclotomic character. We will write $\chi$ for its mod $\varpi$ reduction.

\subsection{Definitions and first properties}
In this section we define (local and global) Selmer groups which will be in use throughout the paper and recall some of their basic properties - for a more detailed treatment we refer the reader to section 5 of \cite{BergerKlosin13}. 
 For $\Sigma$ a finite set of finite places of $\bfQ$ containing $p$ we write $G_{\Sigma}$ for the Galois group of the maximal extension of $\bfQ$ unramified outside of $\Sigma$ and infinity. We write $G_{\bfQ_{\ell}}$ for  the absolute Galois group of $\bfQ_{\ell}$.
Let $M$ be an $\Oo$-module with an $\Oo$-linear action of $G=G_K$ or $G_{\Sigma}$. We call $M$ a $p$-adic $G$-module over $\Oo$ if one of the following holds:
\begin{enumerate}
  \item $M$ is \emph{finitely generated}, i.e. a finitely generated $\bfZ_p$-module and the $G$-action is continuous for the $p$-adic topology
      on $M$;
  \item $M$ is \emph{discrete}, i.e. a torsion $\bfZ_p$-module of finite corank (i.e. $M$ is isomorphic as a $\bfZ_p$-module to
      $(\bfQ_p/\bfZ_p)^r \oplus M'$ for some $r \geq 0$ and some $\bfZ_p$-module $M'$ of finite order) and the $G$-action on $M$ is continuous
      for the discrete topology on $M$;
  \item $M$ is a finite-dimensional $\bfQ_p$-vector space and the $G$-action is continuous for the $p$-adic topology on $M$.
\end{enumerate}

$M$ is both finitely generated and discrete if and only if it is of finite cardinality.

Given a $p$-adic $G_{\Sigma}$-module $M$  we assume that we have a finite-singular structure $\mS$ on $M$ in the sense of \cite{Weston00}, i.e. for each prime $\ell \in \Sigma$  a choice of $\Oo$-submodule $H^1_{f, \mS}(\bfQ_{\ell},M) \subseteq H^1(\bfQ_{\ell},M)$.
We then define two global Selmer groups for $M$:
	\be \label{BKsel} H^1_{f, \mS}(\bfQ,M)={\rm ker}(H^1(G_{\Sigma},M) \to \prod_{\ell \in \Sigma} H^1(\bfQ_{\ell},M)/H^1_{f,\mS}(\bfQ_{\ell},M)\ee
and a ``relaxed'' Selmer group (no conditions at primes $\ell \in \Sigma\setminus\{p\}$)
$$H^1_{\Sigma, \mS}(\bfQ, M):= \ker \left( H^1(G_{\Sigma}, M) \to H^1(\bfQ_{p},M)/H^1_{f, \mS}(\bfQ_p, M)\right).$$

We consider the following local finite-singular structures  $H^1_{f, \mS}(\bfQ_{\ell},M)$  (dropping the subscript $\mS$ for the place at infinity and at $p$, where we fix the choice of finite-singular structure):

We always take $H^1_f(\bfQ_{\ell},M)=0$ for $\ell \mid \infty$.

 For $\ell=p$ we define the \emph{crystalline} local finite-singular structure as follows.
 Let $T \subseteq V$ be a $G_K$-stable $\bfZ_p$-lattice and put $W=V/T$. For $n \geq 1$, put $$W_n=\{x \in W: \varpi^n x =0\}\cong T/\varpi^n T.$$

Following Bloch and Kato \cite{BlochKato90} we define
$H^1_f(\bfQ_{\ell},V)={\rm ker}(H^1(\bfQ_{\ell},V)\to H^1(\bfQ_{\ell},B_{\rm crys} \otimes V))$, denote by $H^1_f(\bfQ_{\ell},T)$ its pullback via the natural map $T \hookrightarrow V$ and let $H^1_f(\bfQ_{\ell}, W)={\rm im}(H^1_f(\bfQ_{\ell},V)\to H^1(\bfQ_{\ell},W))$. Finally, we set $H^1_f(\bfQ_{\ell}, W_n)$ to be the inverse image of $H^1_f(\bfQ_{\ell}, W)$ under the map $H^1(\bfQ_{\ell}, W_n) \to H^1_f(\bfQ_{\ell}, W)$.

For finitely generated $p$-adic $G_{\bfQ_p}$-modules we recall the theory of Fontaine-Laffaille \cite{FontaineLaffaille82}, following the exposition in
\cite{ClozelHarrisTaylor08} Section 2.4.1. Let $\mathcal{M F}_{\Oo}$ (``Dieudonn\'e modules'') denote the
category of finitely generated $\Oo$-modules  $M$ together with a decreasing filtration ${\rm Fil}^i M$ by $\Oo$-submodules which are $\Oo$-direct summands with
${\rm Fil}^0 M=M$ and ${\rm Fil}^{p-1} M=(0)$ and Frobenius  linear maps $\Phi^i:{\rm Fil}^i M \to M$ with $\Phi^i|_{{\rm Fil}^{i+1} M} = p
\Phi^{i+1}$ and $\sum \Phi^i {\rm Fil}^i M=M$. They define an exact, fully faithful covariant functor $\mathbf{G}$
of $\Oo$-linear categories from
$\mathcal{M F}_{\Oo}$ (in their notation $\mathbb{G}_{\tilde v}$ and $\mathcal{M F}_{\Oo,\tilde v}$)  to the category of finitely generated $\Oo$-modules with
continuous action by $G_{\bfQ_p}$. Its essential image is closed under taking subquotients, finite direct sums and contains quotients of lattices in
 short crystalline representations defined as
follows: We call $V$ a continuous finite-dimensional $G_{\bfQ_p}$-representation over $\bfQ_p$  \emph{short crystalline}
 if, for all places $v \mid p$,
${\rm Fil}^0 D=D$ and ${\rm Fil}^{p-1} D=(0)$ for the filtered vector space $D=(B_{\rm crys} \otimes_{\bfQ_p} V)^{G_v}$ defined by Fontaine.

For any $p$-adic  $G_{\bfQ_{p}}$-module $M$ of finite cardinality in the essential image of $\mathbf{G}$ (we will call such modules short crystalline) we define $H^1_f(\bfQ_{p},M)$ as the image of ${\rm Ext}^1_{\mathcal{M
F}_\Oo}(1_{\rm FD}, D)$ in $H^1(\bfQ_{p},M)\cong {\rm Ext}^1_{\Oo[G_{\bfQ_{p}}]}(1,M)$, where $\mathbf{G}(D)=M$ and $1_{\rm FD}$ is the unit filtered Dieudonn\'e
module defined in Lemma 4.4 of \cite{BlochKato90}.

For any complete Noetherian $\bfZ_p$-algebra we define a representation $\rho: G_{\bfQ_p} \to {\rm GL}_n(A)$ to be crystalline if for each Artinian quotient
$A'$ of $A$, $\rho \otimes A'$ lies in the essential image of $\mathbf{G}$.

\begin{rem} \label{DFG22}
Note that for short crystalline representations $V$ we gave two different definitions for $H^1_f(\bfQ_p, W_n)$ (via pullback of $H^1_f(\bfQ_p, W)$ and via the $\bfG$-functor). That these agree follows from the proof of Proposition 2.2 in \cite{DiamondFlachGuo04} or \cite{BergerKlosin13} Lemma 5.3.

\end{rem}

For primes $\ell \neq p$ we define the \emph{unramified} local finite-singular structure on any $p$-adic $G_{\bfQ_{\ell}}$-module $M$ over $\Oo$  as $$H^1_{f, {\rm ur}}(\bfQ_{\ell},M)=H^1_{\rm
ur}(\bfQ_{\ell},M)={\rm ker}(H^1(\bfQ_{\ell},M) \to H^1(\bfQ_{\ell, \rm ur},M)),$$  where $\bfQ_{\ell, \rm ur}$ is the maximal unramified extension of  $\bfQ_{\ell}$.

Let $V$ be a continuous finite-dimensional $G_{\bfQ_{\ell}}$-representation over $\bfQ_p$ and $T \subseteq V$ be a $G_{\bfQ_{\ell}}$-stable $\bfZ_p$-lattice and put $W=V/T$.
Bloch-Kato then define the following \emph{minimal} finite-singular structures on $V$, $T$ and $W$:   $$H^1_{f, {\rm min}}(\bfQ_{\ell},V)=H^1_{\rm ur}(\bfQ_{\ell},V), $$ $$H^1_{f, {\rm min}}(\bfQ_{\ell},T)=i^{-1} H^1_f(\bfQ_{\ell},V) \text{ for } T \overset{i}{\hookrightarrow} V$$ and $$H^1_{f , {\rm min}}(\bfQ_{\ell},W)={\rm im}(H^1_f(\bfQ_{\ell},V) \to H^1(\bfQ_{\ell},W)).$$

\begin{rem} \label{two choices}
Following \cite{Rubin00} Definition 1.3.4 we define $H^1_{f, {\rm min}}(\bfQ_{\ell},W_n)$ as the inverse image of
$H^1_{f, {\rm min}}(\bfQ_{\ell},W)$ under the map $H^1(\bfQ_{\ell},W_n) \to H^1(\bfQ_{\ell},W)$.

By \cite{Rubin00} Lemma 1.3.5 we have $H^1_{f, {\rm min}}(\bfQ_{\ell},W)=H^1_{\rm ur}(\bfQ_{\ell},W)_{\rm div}$ 
 and if $W^{I_{\ell}}$ is divisible we further have $H^1_{f, {\rm min}}(\bfQ_{\ell},W)=H^1_{\rm ur}(\bfQ_{\ell},W)$, $H^1_{f, {\rm min}}(\bfQ_{\ell},T)=H^1_{\rm ur}(\bfQ_{\ell},T)$ and $H^1_{f, {\rm min}}(\bfQ_{\ell},W_n)=H^1_{\rm ur}(\bfQ_{\ell},W_n)$ (for the latter using also \cite{Rubin00} Lemma 1.3.8 and its proof).
\end{rem}

\begin{lemma}[\cite{Weston00} Lemma II.3.1] \label{Sel_funct}
Let $0 \to M' \overset{i}{\to} M \overset{j}{\to} M'' \to 0$ be an exact sequence of $p$-adic $G_{\Sigma}$-modules over $\Oo$. Assume that the modules have been given the induced finite-singular structures, i.e. such that $H^1_{f, \mS} (\bfQ_{\ell}, M'') = j_* H^1_{f, \mS}(\bfQ_{\ell}, M)$ and $H^1_{f, \mS}(\bfQ_{\ell}, M') =
i_*^{-1} H^1_{f, \mS}(\bfQ_{\ell}, M)$ for all ${\ell}$.

Then there is an exact sequence $$0 \to H^0(\bfQ,M') \to H^0(\bfQ,M) \to H^0(\bfQ,M'') \to H^1_{f, \mS}(\bfQ, M') \to H^1_{f, \mS}(\bfQ, M) \to H^1_{f, \mS}(\bfQ, M'').$$
\end{lemma}

The following proposition summarizes the facts we will need about the Selmer groups we have defined:

\begin{prop} \label{torsion coeffs} \begin{enumerate} \item Suppose $W^{G_{\Sigma}}=0$.  Then $H^1_{\Sigma}(\bfQ, W_n) \cong H^1_{\Sigma}(\bfQ, W)[\varpi^n]$ and $H^1_{f, {\rm min}}(\bfQ, W_n) \cong H^1_{f, {\rm min}}(\bfQ, W)[\varpi^n]$. \item If the module $W^{I_{\ell}}$ is divisible for every $\ell \in \Sigma\setminus\{p\}$, then $H^1_{f, {\rm ur}}(\bfQ, W_n)=H^1_{f, {\rm min}}(\bfQ, W_n)$ and   $H^1_{f, {\rm ur}}(\bfQ, W)=H^1_{f, {\rm min}}(\bfQ, W)$. \item If $W$ is short crystalline at $p$ then one can identify $H^1_f(\bfQ_p, W_n)$ as the image of ${\rm Ext}^1_{\mathcal{M
F}_\Oo}(1_{\rm FD}, D)$, where $D$ maps to $W_n$ by the Fontaine-Laffaille functor. \end{enumerate}\end{prop}
\begin{proof} The first statement follows from Lemma \ref{Sel_funct} applied to the sequence $0 \to W[\varpi^n] \to W \overset{\times \varpi^n}{\to} W \to  0$  since the minimal and crystalline finite-singular structures on $W_n$ are induced from the one on $W$ by definition. (This statement is also stated as \cite{Rubin00}, Lemma 1.5.4 (where this Selmer group is denoted $\mS^{\emptyset}(\bfQ, W)$) upon noting that our assumption forces injectivity of the map $H^1(G_{\Sigma}, W_n) \to H^1(G_{\Sigma},W)[\varpi^n]$ by \cite{Rubin00}, Lemma 1.2.2(i)). 

The second statement follows immediately from Remark \ref{two choices}, and the final one from 
Remark \ref{DFG22}. \end{proof}

In the following we will write $H^1_f(\bfQ,M)$ for $H^1_{f, {\rm ur}}(\bfQ, M)$ and $H^1_{\Sigma}(\bfQ,M)$ for $H^1_{\Sigma, {\rm ur}}(\bfQ, M)$  for any $p$-adic $G_{\Sigma}$-module $M$.

\subsection{More properties of Selmer groups}

Assume that $k\geq 2$ is an even integer such that $p>2k-2$.  Note that under this assumption the $p$-adic Galois representation $\rho_f$ associated to an eigenform of weight $2k-2$ of level not divisble by $p$ is short crystalline.  This assumption will be in force throughout the paper.

\begin{prop} \label{invariants} Let
$\rho_f : G_{\Sigma} \to \GL_2(E)$ be the Galois representation attached to a newform $f \in S_{2k-2}(N)$ for $N$ a square-free integer with $p \nmid N$ and $\Sigma=\{ \ell \mid N \} \cup \{p\}$. Suppose also that the residual representation $\ov{\rho}_f$ is  absolutely irreducible and  ramified at every prime $\ell \mid N$. 
Then for $i=1, 2$ \begin{itemize}
\item[(i)] $(\rho_f(i-k)\otimes E/\Oo)^{I_{\ell}}$ is divisible for all $\ell \neq p$
\item[(ii)] $(\rho_f(i-k)\otimes E/\Oo)^{G_{\Sigma}}=0$.
\end{itemize}
\end{prop}
\begin{proof} For the convenience of the reader we include the proof, however one can also consult the proof of Lemma 3.4 in \cite{PollackWeston11}. 

(i)  Set $\varrho:=\rho_f(i-k)$ ($i=1,2$). We choose a basis $\{e_1, e_2\}$ for the space $V=V_{\varrho}$ so that $\varrho$ with respect to that basis  is valued in $\GL_2(\Oo)$. We set $T=\Oo e_1 \oplus \Oo e_2$ (a $G_{\Sigma}$-stable $\Oo$-lattice in $V$) and set $W=V/T$. 

One has by our assumption and by Theorem 3.26(3)(b) of \cite{Hida00} that  \be\label{min for rho 2} \rho_f|_{G_{\bfQ_{\ell}}} \cong_E \bmat \psi\epsilon &* \\ & \psi \emat, \quad \rho_f|_{I_{\ell}} \neq 1,\ee where $\psi: G_{\bfQ_{\ell}} \to E^{\times}$ denotes the unramified character sending $\Frob_{\ell}$ to $a_{\ell}(f)$. 
So,  there exists $A\in \GL_2(E)$ such that $$A\varrho|_{G_{\ell}}A^{-1} = \bmat \psi\epsilon^{i-k+1} & * \\ & \psi\epsilon^{i-k}\emat.$$ First let's note that $\varrho$ is only tamely ramified at $\ell$ and so the image of $I_{\ell}$ is pro-cyclic. Suppose that $g \in I_{\ell}$ is such that $\varrho(g)$ topologically generates the image of $I_{\ell}$. Conjugating $\varrho(g)$ by a matrix of the form $\diag(u\varpi^n, 1)$ for a suitable $n$ and unit $u$ we may assume that $A\varrho(g)A^{-1} = \bmat 1&1\\ & 1 \emat\in \GL_2(\Oo).$ Write $A=\bmat a&b\\ c&d \emat$. By enlarging $E$ (so that $\sqrt{\det A}\in E$) we may also assume that $\det A=1$. 
 So, $$\varrho(g) = A^{-1}\bmat 1&1\\ & 1\emat A = \bmat 1+cd & d^2 \\ -c^2 & 1-cd\emat.$$ Since $\varrho$ is valued in $\GL_2(\Oo)$ we must have $c^2\in \Oo$ and $d^2\in \Oo$. Since $\Oo$ is integrally closed in $E$ this implies $c, d \in \Oo$. 

Since $\ov{\rho}_f$ is ramified at $\ell$ so is $\ov{\varrho}$, i.e., we must have that either $c \in \Oo^{\times}$ or $d \in \Oo^{\times}$. Let $v\in V$. Then $v +T \in W^{I_{\ell}}$ if and only if $g \cdot v - v \in T$. Writing $v$ in the basis $\{e_1, e_2\}$ as $v=\bmat x\\ y\emat$ the above relation assumes the form: $\bmat x \\ y \emat + \bmat \Oo \\ \Oo \emat \in W^{I_{\ell}}$ if and only if 
 $$\bmat cdx +d^2y\\ -c^2x - cd y \emat=\varrho(g)\bmat x\\ y \emat -\bmat x \\ y \emat \in \bmat \Oo \\ \Oo \emat$$ which happens if and only if $$d(cx+dy) \in \Oo \quad \textup{and}\quad -c(cx+dy) \in \Oo.$$ Now suppose $c \in \Oo^{\times}$. Then from the second relation we get $cx+dy \in \Oo$. Since this condition also implies that the first expression is in $\Oo$, we get that  in this case $\bmat x\\y \emat +T \in W^{I_{\ell}}$ if and only if $cx+dy\in \Oo$. Similarly if $d \in \Oo^{\times}$, then we again get the exact same condition. Hence we conclude that $$W^{I_{\ell}} = \left\{\bmat x \\ y \emat +\bmat \Oo \\ \Oo \emat \mid x,y\in E,  cx+dy \in \Oo\right\}.$$ Suppose $c\in \Oo^{\times}$. Then $$W^{I_{\ell}} = \left\{\bmat -c^{-1}dy \\ y \emat +\bmat \Oo \\ \Oo \emat \mid y \in E\right\}.$$  Consider the map from the module on the right to $E/\Oo$ given by $$\bmat -c^{-1}dy \\ y \emat +\bmat \Oo \\ \Oo \emat \mapsto y+\Oo.$$ It is clearly a map of $\Oo$-modules and clearly surjective. Suppose that $\bmat -c^{-1}dy \\ y \emat +\bmat \Oo \\ \Oo \emat \in \ker$. Then $y \in \Oo$, which implies that $c^{-1}dy\in \Oo$, so the map is also injective, hence an isomorphism. Analogous map works in the case when $d\in \Oo^{\times}$. Since $E/\Oo$ is divisible this proves that $W^{I_{\ell}}$ is divisible.

(ii)
Let $T, W$ be as in the proof of (i). Clearly if $v \in W^{G_{\Sigma}}$, then some $\Oo$-multiple of $v$ is a non-zero element in $W[\varpi]$. Hence it is enough to show that $W[\varpi]^{G_{\Sigma}}=0$. However, we have (see e.g.   \cite{Rubin00}, Lemma 1.2.2(iii)) $W[\varpi] \cong T/\varpi T \cong \ov{\varrho}$, so the claim follows from the irreducibility of $\ov{\rho}_f$. 
\end{proof}

\begin{cor} \label{l not equal to p} Let $\rho_f$ be as in Proposition \ref{invariants}. For $i=1,2$ one has $$H^1_{f, {\rm min}}(\bfQ, \ov{\rho}_f(i-k)) = H^1_{f, {\rm ur}}(\bfQ, \ov{\rho}_f(i-k)).$$
\end{cor}
\begin{proof} This follows from Propositions \ref{invariants} and \ref{torsion coeffs}(1)+(2). \end{proof}

\begin{prop} \label{Selmer groups 2} Let $\rho_f : G_{\Sigma} \to \GL_2(E)$ be the Galois representation attached to a newform $f \in S_{2k-2}(N)$ with $p \nmid N$.  Suppose also that the residual representation $\ov{\rho}_f$ is   absolutely irreducible and ramified at every prime $\ell \mid N$.  One has $H^1_{\Sigma}(\bfQ, {\rho}_f(2-k)\otimes E/\Oo) = H^1_f(\bfQ,{\rho}_f(2-k)\otimes E/\Oo)$ if
 $N$ is square-free, $\Sigma=\{\ell \mid N\}\cup \{p\}$, and $p \nmid 1+w_{f, \ell}\ell$ for all primes $\ell \mid N$. Here $w_{f,\ell}\in \{1, -1\}$ is the eigenvalue of the Atkin-Lehner involution at $\ell$. 
 \end{prop}
\begin{rem} Note that when $N$ is prime $w_{f,N}=(-1)^{k-1}W_f$, where  $W_f \in \{1, -1\}$ is the sign of the functional equation of $f$ (see e.g. \cite{Shimura71} Theorem 3.66). \end{rem}

\begin{proof} We use Lemma 5.6 in \cite{BergerKlosin13} with $V=\rho_f(2-k)$. Then $V^*=V^{\vee}(1)=V$. 

By \eqref{min for rho 2} we have that $(V^*)^{I_{\ell}}=\psi (3-k)$
and hence we get (using the notation of \cite{BergerKlosin13}) $$P_{\ell}(V^*,1)=\det(1-\Frob_{\ell}|_{(V^*)^{I_{\ell}}})=1-a_{\ell}(f)\ell^{3-k}.$$ By Theorem 3 in \cite{AtkinLehner} (but with Hecke action normalized as in \cite{Miyake89}, see \cite{Miyake89} Theorem 4.6.17(2) for the correct power of $\ell$ in our normalization) we know that $a_{\ell}(f)=-\ell^{k-2}w_{f, \ell}$. This shows that under the conditions of the Proposition we have $\val_p(P_{\ell}(V^*,1))=0$ for all $\ell \nmid N$. 
To be able to use Lemma 5.6 in \cite{BergerKlosin13}, we also need to know that $W^{I_{\ell}}$ is divisible, 
 where $W=V/T$ with $V$ as above  (i.e, the two dimensional vector space over $E$ on which $G_{\Sigma}$ acts via $\rho_f(2-k)$) and $T \subset V$ a $G_{\Sigma}$-stable $\Oo$-lattice.  However, this follows from Proposition \ref{invariants}(i). We note here that the order of the Selmer group does not depend on the choice of $T$ since $\rho$ is  absolutely irreducible (cf. e.g., Lemma  9.7 in \cite{Klosin09}).
\end{proof}

\begin{prop} \label{class groups} One has $$H^1_f(\bfQ,\bfF(1)) =  H^1_{f}(\bfQ, \bfF(-1)) =   0.$$ \end{prop}
\begin{proof}For the latter Selmer group  this is essentially Proposition 6.16 along with Theorem 6.17 in \cite{Washingtonbook} which tell us that the $p$-part of the class group of the splitting field of $\chi^{-1}$ is trivial (as the numerator of $B_2$ (the second Bernoulli number) is 1). 
This implies that $H^1_{f, {\rm min}}(\bfQ, W_1)=0$ for $W=\bfQ_p/\bfZ_p(-1)$.  By Proposition \ref{torsion coeffs} we also have that $H^1_{f, {\rm min}}(\bfQ, W_1)=H^1_{f, {\rm ur}}(\bfQ, W_1)=H^1_{f}(\bfQ, W_1)$ since $W^{I_{\ell}}=0$ for all $\ell \neq p$.

On the other hand \cite{Rubin00} Proposition 1.6.4(ii)  tells us that $H^1_f(\bfQ, \bfQ_p/\bfZ_p(1))$ fits into a short exact sequence
$$0 \to \bfZ^{\times} \otimes \bfQ_p/\bfZ_p \to H^1_f(\bfQ, \bfQ_p/\bfZ_p(1)) \to {\rm Cl}(\bfQ) \to 0.$$ This implies that $H^1_f(\bfQ, \bfQ_p/\bfZ_p(1))=0$, and therefore $H^1_f(\bfQ, \bfF(1))=H^1_f(\bfQ, \bfQ_p/\bfZ_p(1))[p]=0$.
\end{proof}

\subsection{Computing the non-critical Selmer group} \label{noncritical}

We will be interested in bounding from above the order of $H^1_f(\bfQ, \rho_f(1-k) \otimes E/\Oo)$. In the case that $k=2$ the motif $\rho_f(1-k)$ is not critical in the sense of Deligne. However, we will show how it can still be bounded by a certain $p$-adic $L$-value using Iwasawa theory for $p \neq 3$.

Let $\omega$ be the Teichmüller lift of $\chi$. Suppose $f \in S_2(\Gamma_0(N))$ (with $p \nmid N$) is a newform ordinary at $p$ and that $\ov{\rho}_f$ is absolutely irreducible and ramified at all primes $\ell \mid N$. We will denote by $L_p^{\rm an}(f, \omega^{-1}) \in \Oo[[T]]$ the Manin-Vishik $p$-adic $L$-function for $f \otimes \omega^{-1}$ (see e.g. \cite{Kato04} Theorem 16.2 where the $p$-adic $L$-function is defined for  the $p$-stabilisation of $f$  as an element of $\tilde \Lambda=\Oo[\Delta][[T]]$; we take its projection onto the component of $\tilde \Lambda$ where $\Delta$ acts by $\omega^{-1}$). 
\begin{prop} \label{prop3.8}
For $p\neq 3$ we have $$\val_p(\#H^1_f(\bfQ, \rho_f(-1)\otimes E/\Oo)) \leq \val_p \#(\Oo/L_p^{\rm an}(f, \omega^{-1}, T=p)).$$ 
\end{prop}
\begin{proof}
Write $T_f$ for a choice of a $G_{\Sigma}$-invariant lattice inside the Galois representation $V_f=\rho_f :G_{\Sigma}\to \GL_2(E)$. Since $f$ is ordinary there is a unique $D_p$-stable line   $V_f^+$  of $V_f$ on which $I_p$ acts via $\epsilon$.  Set $T_f^+:= T_f \cap V_f^+$. Following \cite{SkinnerUrban14} (p. 34) for any continuous character $\xi: G_{F,\Sigma} \to \Oo^{\times}$ we define the Greenberg Selmer group $$\Sel_{F, E}^{\Sigma}(f, \xi):= \ker\left\{H^1(G_{F,\Sigma}, T_f\otimes \xi\otimes E/\Oo)\to  H^1(I_p, (T_f/T_f^+)\otimes \xi \otimes E/\Oo)\right\},$$  where $F=\bfQ$ or $F=\bfQ_{\infty}$, the cyclotomic $\bfZ_p$-extension of $\bfQ$. We are interested in the case where $\xi=\omega^{m-1}\epsilon^{-m}$ for $m=1$, while Skinner and Urban are mainly interested in the case where $\xi=\omega^m \epsilon^{-m}$. We set $X_{F, E}^{\Sigma}(f, \xi):=\Hom_{\Oo}(\Sel_{F, E}^{\Sigma}(f, \xi), E/\Oo)$ to be the Pontryagin dual. By a theorem due to Kato (\cite{SkinnerUrban14}, Theorem 3.15) we know that $X_{\bfQ_{\infty}, E}^{\Sigma}(f, \xi)$ is a torsion $\Lambda$-module, where $\Lambda=\Oo[[\Gal(\bfQ_{\infty}/\bfQ)]]=\Oo[[T]]$. Set $Ch_{\bfQ_{\infty}, E}^{\Sigma}(f, \omega^{-1})$ to be the characteristic ideal of $X_{\bfQ_{\infty}, E}^{\Sigma}(f, \omega^{-1})$. The Main Conjecture of Iwasawa Theory states that $Ch_{\bfQ_{\infty}, E}^{\Sigma}(f, \omega^{-1})$ is a principal ideal generated by $L_p^{\rm an}(f, \omega^{-1}, T)$ (cf. Conjecture 3.24 in \cite{SkinnerUrban14}). By Theorem 17.14 in \cite{Kato04} (but see also Theorem 3.25 in \cite{SkinnerUrban14}, whose notation is consistent with ours) the ideal $Ch_{\bfQ_{\infty}, E}^{\Sigma}(f, \omega^{-1})$ divides $L_p^{\rm an}(f, \omega^{-1}, T)$ in $\Lambda$. Note that the assumption in Theorem 17.14 in \cite{Kato04} that the image of $\rho_f$ contains $\SL_2(\bfZ_p)$ can be replaced by $\ov{\rho}_f$ being ramified at some $\ell \| N$, as explained in \cite{Skinner16} pages 187/8. 

To relate this to our $H^1_f(\bfQ, \rho_f(-1)\otimes E/\Oo)$ we need a control theorem which was proved by Skinner and Urban in the case $\xi=\omega^m\epsilon^{-m}$ (Corollary 3.21(i) in \cite{SkinnerUrban14} which should say $m \neq 0$ instead of $m \neq 1$). Let us only indicate here what changes are needed to adapt it to the case $\xi=\epsilon^{-1}=\omega^{-1} \cdot (\omega/\epsilon)^m$ with $m=1$. 
\begin{thm} [Control Theorem] \label{control} One has $$\#X_{\bfQ, E}^{\Sigma}(f, \epsilon^{-1}) \mid \# \Lambda/(T-p, Ch_{\bfQ_{\infty}, E}^{\Sigma}(f, \omega^{-1})).$$ \end{thm}
\begin{proof} Skinner and Urban prove the corresponding statement (for $\xi=\omega^m\epsilon^{-m}$) by a series of results culminating in Proposition 3.20 which is then used to prove Corollary 3.21 (here we only care about one divisibility as stated above instead of equality). The assumption that $f$ be $p$-stabilized in \cite{SkinnerUrban06} Section 3.3.13 is not necessary. The first result is Proposition 3.10. The proof of the only part  relevant for us, the injectivity of the map $\Sel_{\bfQ, E}^{\Sigma}(f, \epsilon^{-1}) \to \Sel_{\bfQ_{\infty}, E}^{\Sigma}(f, \epsilon^{-1})^{G_{\Sigma}}$,
carries over unchanged to our situation. As the next step we need an analogue of Proposition 3.20 (stating that $X_{\bfQ_{\infty}, E}^{\Sigma}(f, \omega^{-1})$ has no non-zero pseudo-null $\Lambda$-submodules) for which it is sufficient to prove corresponding analogues of Lemmas 3.18 and 3.19.  Let us explain what changes are needed to the proof of Lemma 3.18 which for us only holds for $p\neq 3$.  While in \cite{SkinnerUrban14} the module denoted there by $M^-[x]|_{I_p}$ is isomorphic to $E/\Oo(\omega^m\epsilon^{-m})$
we have  $M^-[x]|_{I_p} \cong E/\Oo(\omega^{m-1}\epsilon^{-m})$ (and $m=1$). The assumption that $p \neq 3$ is needed in our case to ensure that  $M^-[\fm]^*(1)=\bfF(\omega^2)$ is ramified at $p$ (note $\epsilon = \omega$ mod $\varpi$ and $\fm=(x, \varpi)$), where $*$ denotes the Pontryagin dual. Also in our case $M^-[x]^{I_p}=0$, hence in particular finite.  Lemma 3.19 follows the same way as in \cite{SkinnerUrban14}. Having all this, Corollary 3.21 is also proved the same way as in \cite{SkinnerUrban14}. \end{proof}

We continue with the proof of Proposition \ref{prop3.8}. As a corollary of the control theorem we obtain the following inequality: $$\#X_{\bfQ, E}^{\Sigma}(f, \epsilon^{-1}) \leq \# \Oo/L_p^{\rm an}(f, \omega^{-1}, T=p)$$ (cf. Theorem 3.36 in \cite{SkinnerUrban14}). Note that Skinner and Urban work within the critical range of the relevant $L$-function, while we are outside of this range, so in particular we do not relate the $p$-adic $L$-value to a classical one. 

Let us now show that $\#X_{\bfQ, E}^{\Sigma}(f, \epsilon^{-1}) \geq \#H^1_f(\bfQ, \rho_f(-1)\otimes E/\Oo))$. First note that  $\#X_{\bfQ, E}^{\Sigma}(f, \epsilon^{-1})=\# \Sel_{\bfQ, E}^{\Sigma}(f, \epsilon^{-1})$ since the group on the left is finite. We will show that $\Sel_{\bfQ, E}^{\Sigma}(f, \epsilon^{-1}) \supset H^1_f(\bfQ, \rho_f(-1)\otimes E/\Oo)$. The conditions at $\ell \neq p$ defining both groups coincide by combining   Propositions \ref{invariants}(i) and  \ref{torsion coeffs}(2), so it is enough to compare conditions at $p$. There we have $H^1_f(\bfQ_p, \rho_f(-1)\otimes E/\Oo)\subset H^1_g(\bfQ_p, \rho_f(-1)\otimes E/\Oo)$ which follows from (3.7) in \cite{BlochKato90} by taking images in the cohomology with divisible coefficients. On the other hand the diagram at the bottom of p. 87 in \cite{Ochiai00} tells us that $H^1_g(\bfQ_p, \rho_f(-1)\otimes E/\Oo)$ injects into \begin{eqnarray*} H^1_{\rm Gr} (\bfQ_p, \rho_f(-1)\otimes E/\Oo) &:=& \ker(H^1(\bfQ_p, T_f(-1)\otimes E/\Oo)\to H^1(I_p, T_f(-1)\otimes E/\Oo))\\&\hookrightarrow& \ker(H^1(\bfQ_p, T_f(-1)\otimes E/\Oo)\to H^1(I_p, (T_f/T_f^+)(-1)\otimes E/\Oo)).\end{eqnarray*} Indeed, $H^1_g(\bfQ_p, \rho_f(-1)\otimes E/\Oo)$ is defined as the image of $H^1_g(\bfQ_p, V_f(-1))$ in $H^1(\bfQ_p, T_f(-1)\otimes E/\Oo)$ and according to the diagram the defining map factors through $H^1_{\rm Gr} (\bfQ_p, \rho_f(-1)\otimes E/\Oo)$. The corresponding inclusions of global Selmer groups follow. 
\end{proof}

\section{Deformations} \label{Deformations}

Let $\rho: G_{\Sigma} \to \GL_2(\bfF)$ be a continuous  absolutely irreducible odd Galois representation of determinant $\chi^{2k-3}$, short crystalline at $p$ and such that for every prime $\ell \in \Sigma\setminus\{p\}$ one has $\rho|_{I_{\ell}} \cong \bmat 1&*\\ 0 & 1\emat \neq I_2$. 
Consider a residual (short crystalline) representation $$\ov{\sigma}:= \bmat \chi^{k-2} & a &b \\ & \rho & c \\ && \chi^{k-1}\emat: G_{\Sigma} \to \GL_4(\bfF),$$ where $a,c$ are assumed to be non-trivial classes in $H^1(\bfQ, \rho(1-k))$ (no assumption on $b$).

\begin{rem}\label{identification}
We have $$ \Hom(\rho, \chi^{k-2}) =\rho^{\vee} \otimes \chi^{k-2} =\rho \otimes \chi^{3-2k} \otimes \chi^{k-2} = \rho(1-k)=\Hom(\chi^{k-1}, \rho),$$ so $a$ and $c$ are indeed
both classes in $H^1(\bfQ, \rho(1-k))$. 
\end{rem}

Set $N_1=\bmat 0 \\ & 0 &1\\ &&0 \\ &&&0 \emat$ and for any prime $\ell \neq p$ let $t_p :I_{\ell} \to \bfZ_p(1)$  be the tame character. 
 We will always assume that $\ov{\sigma}$ as above is \emph{semi-abelian} by which we mean that for each $\ell \in \Sigma \setminus \{p\}$ one has 
\be \label{admissible} \ov{\sigma}|_{I_{\ell}} \cong  \exp (t_pN_1).\ee 
\begin{lemma} \label{form res} Let $x \in G_{\Sigma}$ be such that $\ov{\sigma}(x)$  generates the image of $I_{\ell}$. Then $\ov{\sigma}$ is isomorphic to a representation of the form $$\ov{\sigma}'=\bmat   \chi^{k-2} &*_1 & *_2 \\ & \rho & *_3 \\ && \chi^{k-1} \emat$$ and such that $$\ov{\sigma}'(x)-I_4=\bmat 0&0&0&0\\ 0&0&*&0\\0&0&0&0\\0&0&0&0\emat\neq 0,$$ i.e., the matrix conjugating $\ov{\sigma}|_{I_{\ell}}$ to $\exp(t_{\ell}N_1)$ does not change the matrix form of $\ov{\sigma}$. \end{lemma}
\begin{proof} Without loss of generality assume that $\rho(x)=\bmat 1&\alpha \\ 0 & 1\emat $ with $\alpha \neq 0$. Write $\alpha_{i,j}$ for the $(i,j)$-entry of $\ov{\sigma}$. 
Then $$\ov{\sigma}(x)-I_4 = \bmat 0&\alpha_{1,2}(x) & \alpha_{1,3}(x) & \alpha_{1,4}(x) \\ 0&0 &\alpha&\alpha_{2,4}(x) \\ 0&0&0&\alpha_{3,4}(x) \\ 0&0&0&0\emat.$$
Since ${\rm rank}(\ov{\sigma}(x)-I_4)={\rm rank} (N_1)=1$, we must also have that $M:= \bmat \alpha_{1,3}(x) & \alpha_{1,4}(x) \\ \alpha&\alpha_{2,4}(x)\emat$ has rank 1.
Thus there exists $\beta \in \bfF$ such that $\alpha_{1,3}(x)=\beta \alpha$ and $\alpha_{1,4}(x)=\beta \alpha_{2,4}(x)$. 
Hence we have $$AMD^{-1}=\bmat 0 & 0\\ \alpha & 0 \emat, \quad \textup{where} \quad A=\bmat 1 & -\beta \\ 0 & 1 \emat, D^{-1}=\bmat 1 & -\alpha_{2,4}(x)/\alpha\\ 0&1\emat\in \GL_2(\bfF).$$ 
Thus  $$\bmat A\\ &D \emat (\ov{\sigma}(x)-I_4) \bmat A\\ &D \emat^{-1} = \bmat 0&\alpha'_{1,2}(x) & 0& 0 \\ 0&0 &\alpha&0 \\ 0&0&0&\alpha'_{3,4}(x) \\ 0&0&0&0\emat.$$ However, since the last matrix has rank one and $\alpha \neq 0$, we must have $\alpha'_{1,2}(x)=\alpha'_{3,4}(x)=0$. Furthermore note that \begin{multline}\ov{\sigma}':=\bmat A\\ &D \emat \ov{\sigma} \bmat A\\ &D \emat^{-1} \\
=\bmat A\\ &D \emat 
 \bmat \chi^{k-2}&* &*\\ &\rho &*\\ &&\chi^{k-1}\emat\bmat A\\ &D \emat^{-1} =\bmat A\\ &D \emat 
 \bmat \chi^{k-2}&*' &*'\\ &\rho &*'\\ &&\chi^{k-1}\emat\bmat A\\ &D \emat^{-1}.\end{multline} 
\end{proof}

\begin{cor}\label{splitting at ell} For $\ell \in \Sigma\setminus\{p\}$ we have $\ov{\sigma}|_{I_{\ell}} \cong 1 \oplus \rho|_{I_{\ell}} \oplus 1.$ In particular, the extensions given by the entries $a$ and $c$ of $\ov{\sigma}$ are both split when restricted to $I_{\ell}$. 
\end{cor}
Let $\textup{LCN}(E)$ be the category of local complete Noetherian $\Oo$-algebras with residue field $\bfF$. 
\begin{definition} 
For an object $A$ of ${\rm LCN}(E)$ we will say that $\sigma: G_{\Sigma} \to \GL_4(A)$ is \emph{minimal} if 
 $\sigma|_{I_{\ell}}\cong_{{\rm Tot}(A)} \exp (t_pN_1)$ for every prime $\ell \in \Sigma \setminus \{p\}$. Here $\rm{Tot}(A)$ denotes the total ring of fractions of $A$.
 \end{definition}
 Note that any minimal $\sigma: G_{\Sigma} \to \GL_4(A)$ is tamely ramified at all $\ell \in \Sigma \setminus\{p\}$, so $\sigma (I_{\ell})$ is a pro-cyclic group for all such $\ell$.

We consider the deformation problem for $\ov{\sigma}$ where to every object $A$ of ${\rm LCN}(E)$ we assign the set of strict equivalence classes of representations $\sigma: G_{\Sigma} \to \GL_4(A)$ which are short crystalline, have determinant $\epsilon^{4k-6}$, are minimal  and reduce to $\ov{\sigma}$ modulo the maximal ideal $\fm_A$ of $A$.

\begin{thm} \label{representability} The above deformation problem is representable. \end{thm}
\begin{proof} It is easy to check that $\ov{\sigma}$ has scalar centralizer, so  by e.g. \cite{ClozelHarrisTaylor08} Proposition 2.2.9  it is enough to verify that short crystallinity and minimality are deformation conditions. The former is due to Ramakrishna (cf. Theorem 1.1 of \cite{Ramakrishna93}). The latter can be proven  similar to \cite{Mazur97} Section 29.  \end{proof}
We will write $R'$ for the universal deformation ring
and $(\sigma')^{\rm univ}: G_{\Sigma} \to \GL_4(R')$ for the universal deformation.

Consider  the following antiinvolution $\tau: R'[G_{\Sigma}] \to R'[G_{\Sigma}]$ given by $\tau(g) = \epsilon^{2k-3}(g) g^{-1}$. Clearly $\tau$ induces the following permutation on the set $\mI$ of Jordan-Hölder factors of $\ov{\sigma}$: $$ \chi^{k-1} \mapsto \chi^{k-2}, \quad \rho \mapsto \rho, \quad \chi^{k-2} \mapsto \chi^{k-1}.$$ 
For $A \in \textup{LCN}(E)$ we will say that a short crystalline deformation $\sigma: G_{\Sigma} \to \GL_4(A)$ is $\tau$-self-dual if $\tr \sigma = \tr \sigma \circ \tau$. 
 Note that for semisimple $\sigma$ and fields $A$ this is equivalent to $\sigma^{\vee} \cong \sigma(3-2k)$.

\begin{prop} \label{reprsd} The functor assigning to an object $A \in
\textup{LCN}(E)$ the set of strict equivalence classes of 
 $\tau$-self-dual short crystalline  (at $p$), minimal
deformations to $\GL_{4}(A)$ is representable by the quotient of
$R'$ by the ideal generated by $\{\tr (\sigma')^{\rm univ}(g)-\tr (\sigma')^{\rm univ}(\tau(g)) \mid g \in G_{\Sigma} \}$. We will denote this quotient by
$R$ and will write $\sigma^{\rm univ}$ for the corresponding universal deformation. \end{prop}

We write $R^{\rm red}$ for the quotient of $R$ by its nilradical and $\sigma^{\rm red}$ for the corresponding (universal)
deformation, i.e., the composite of $\sigma^{\rm univ}$ with $R\twoheadrightarrow R^{\rm red}$.

\section{Lattice} \label{Lattice}

\begin{thm} \label{centralizer} Let $E$ be a finite extension of $\bfQ_p$ with valuation ring $\Oo$, uniformizer $\varpi$ and $\bfF=\Oo/\varpi\Oo$. Let $\sigma: G_{\Sigma} \to \GL_n(E)$ be an absolutely irreducible Galois representation. Suppose there exists a $G_{\Sigma}$-stable $\Oo$-lattice $\Lambda$ in the space of $\sigma$ such that $$\ov{\sigma}_{\Lambda} \cong \bmat \rho_1 & * \\ 0 & \rho_2\emat$$ with $\rho_i: G_{\Sigma} \to \GL_{n_i}(\bfF)$ (for $i=1,2$) having scalar centralizer,  $\rho_1$ semisimple  and such that none of the irreducible constituents of $\rho_1^{\rm ss}$ is isomorphic to any of the irreducible constituents of $\rho_2^{\rm ss}$. Then there exists a $G_{\Sigma}$-stable lattice $\Lambda'$ in the space of $\sigma$ such that $$\ov{\sigma}_{\Lambda'} \cong \bmat \rho_1 & * \\ 0 & \rho_2\emat \not\cong \rho_1 \oplus \rho_2.$$  \end{thm}
Let us note that if $\rho_2$ is semisimple, then Theorem \ref{centralizer} follows from a generalization of Theorem in \cite{Urban99} - cf. Remark (a) in \cite{Urban99}.  

 We begin the proof of Theorem \ref{centralizer} with the following lemma.
\begin{lemma} \label{aux1} Let $\rho_1, \rho_2$ be as in Theorem \ref{centralizer}. Assume that $\ov{\sigma}: G_{\Sigma} \to \GL_n(\bfF)$ given by $$\ov{\sigma} = \bmat \rho_1 & f \\ & \rho_2\emat$$ is isomorphic to $\rho_1 \oplus \rho_2$. Then there exists a matrix $$M=\bmat I_{n_1} & B \\ & I_{n_2} \emat \in \GL_n(\bfF)$$ such that $M\ov{\sigma} M^{-1} = \bmat \rho_1 \\ & \rho_2\emat.$ \end{lemma}

\begin{proof} By assumption there exists $h=\bmat A & B \\ C & D \emat \in \GL_n(\bfF)$ (with correct dimensions) such that \be \label{eq343} h\bmat \rho_1 & f \\ & \rho_2 \emat h^{-1} = \bmat \rho_1 & 0\\ & \rho_2 \emat\ee  which implies that $\rho_2 C = C \rho_1$, i.e., that $C \in \Hom_{G_{\Sigma}}(\rho_1 , \rho_2)$. The image of $C$ is a quotient of $\rho_1$ which is a submodule of $\rho_2$. However,  by our assumption on the irreducible constituents of $\rho_1^{\rm ss}$ and $\rho_2^{\rm ss}$, this image must be zero, so $C=0$.  Then (\ref{eq343}) implies that $A$ centralizes $\rho_1$ and $D$ centralizes $\rho_2$, hence $A=\alpha I_{n_1}$ and $D=\delta I_{n_2}$ with $\alpha, \delta \in \bfF^{\times}$. By simply scaling $h$ by $\alpha^{-1}$ we may assume without loss of generality that $\alpha=1$. Finally (with $C=0$), the $B$-entry of $h\bmat \rho_1 & f \\ & \rho_2 \emat h^{-1} $ equals \be \label{eq4445} -\alpha\rho_1 \alpha^{-1}B\delta^{-1}+\alpha f\delta^{-1}+ B \rho_2 \delta^{-1}= \rho_1 B \delta^{-1} + f \delta^{-1} + B \rho_2 \delta^{-1}=0.\ee From this we see that the expression in (\ref{eq4445}) equals zero regardless of the value of $\delta$, so we may assume that $\delta=1$. 
\end{proof}

\begin{proof}[Proof of Theorem \ref{centralizer}] In this we follow Ribet \cite{Ribet76}, the proof of
 Proposition 2.1. Let the notation  be as in the statement of the Theorem. Suppose no $\Lambda'$ exists. Then by Lemma \ref{aux1} there exists $h=\bmat I_{n_1} & B \\ & I_{n_2}\emat$ (where $B$ here is any lift to $M_{n_1 \times n_2}(\Oo)$ of $B$ in lemma) such that for all $g \in G_{\Sigma}$ one has $$h\sigma_{\Lambda}(g)h^{-1} = \bmat \alpha & \varpi \beta \\ \varpi \gamma & \delta \emat  \quad \textup{with $\alpha \equiv \rho_1(g)$ mod $\varpi$, $\delta \equiv \rho_2(g)$ mod $\varpi$}$$ and $\alpha, \beta, \gamma, \delta$ 
 with entries  in $\Oo$. Then $$\tau_1(g):= \bmat I_{n_1} \\ & \varpi I_{n_2}\emat h\sigma_{\Lambda}(g)h^{-1}\bmat I_{n_1} \\ & \varpi I_{n_2}\emat^{-1} = \bmat \alpha & \beta \\ \varpi^2 \gamma & \delta \emat.$$ The reduction of $\tau_1$ still satisfies the assumptions of Lemma \ref{aux1}, hence by that lemma there exists $h_1=\bmat I_{n_1} & B_1 \\ & I_{n_2}\emat\in \GL_n(\Oo)$ such that $h_1 \tau_1 h_1^{-1} = \bmat \alpha_1 & \varpi \beta' \\ \varpi^2 \gamma' & \delta'\emat $, where $\alpha', \beta', \gamma'=\gamma, \delta$ 
 have entries  in $\Oo$. We continue this way as Ribet does to conclude in the end that $\sigma$ itself is reducible, which leads to a contradiction. This finishes the proof of Theorem \ref{centralizer}. 
\end{proof}

\begin{cor}\label{latt}  Let   $\sigma: G_{\Sigma} \to \GL_4(E)$ be an  absolutely irreducible short crystalline  Galois representation. Suppose  that $\ov{\sigma}^{\rm ss} = \chi^{k-2} \oplus \rho \oplus \chi^{k-1}$ with $\rho$ absolutely irreducible and  that $\sigma|_{I_{\ell}} \cong_E \exp(t_p N_1)$ for all $\ell \in \Sigma \setminus \{p\}$.  
 Then there exists  a $G_{\Sigma}$-stable lattice $\Lambda$ in the space of $\sigma$ such that $$\ov{\sigma}_{\Lambda} = \bmat \chi^{k-2} &*_1 & *_2 \\ & \rho & *_3 \\ && \chi^{k-1} \emat$$ is semi-abelian with $*_3$ a non-split extension of $\chi^{k-1}$ by $\rho$ 
and $*_1$ a non-trivial extension of $\rho$ by $\chi^{k-2}$. \end{cor}

\begin{proof} 
Applying Theorem 6.1 in \cite{Brown11} with $\mR = \Oo$ and $I=\varpi \Oo$ (note that $\mT\cong \Oo$ because any lattice in $E$ is isomorphic to $\Oo$) we get that there exists a lattice $\Lambda'$ in the space of $\sigma$ such that $$\ov{\sigma}_{\Lambda'} = \bmat \chi^{k-2} & 0 & *_2 \\ & \rho & *_3 \\ && \chi^{k-1} \emat \not\cong \chi^{k-2} \oplus \rho \oplus \chi^{k-1}.$$ We now claim that  $\ov{\sigma}_{\Lambda'}$ cannot be equivalent to a representation of the same shape where $*_3=0$.
Indeed, suppose it were, then $\ov{\sigma}_{\Lambda'} \cong \rho \oplus \bmat \chi^{k-2} & *_2 \\ & \chi^{k-1} \emat$. Let $\ell \in \Sigma \setminus \{p\}$. If $x \in G_{\Sigma}$ is such that $\sigma(x)$ generates the image of $I_{\ell}$, then the rank of $\sigma(x)-I_4$ must be one. Since $\sigma_{\Lambda'}\cong_E \sigma$, we also must have ${\rm rank}(\sigma_{\Lambda'}(x)-I_4)=1$, and so also ${\rm rank}(\ov{\sigma}_{\Lambda'}(x)-I_4)=1$. Let us prove this last implication. Since $\rho|_{I_{\ell}}\neq 1$, we must have ${\rm rank}(\ov{\sigma}_{\Lambda'}(x)-I_4)\geq 1$. So, we need to prove a rank one matrix cannot reduce mod $\varpi$ to a matrix of a higher rank. Let $A$ be a rank one matrix with entries in $\Oo$. Then every row is a scalar multiple of the first non-zero row. These scalars are of the form $u \varpi^n$, where $n \in \bfZ$. Pick a row for which $n$ is minimal and by making this row first (permutation matrix has entries in $\Oo$), we may now assume that all $n \geq 0$, i.e, that all the scalars are in $\Oo$. Thus every row of the reduction of $A$ mod $\varpi$ is a scalar multiple of the first row. This establishes that ${\rm rank}(\ov{\sigma}_{\Lambda'}(x)-I_4)\leq 1$.

However, if $\ov{\sigma}_{\Lambda'} \cong \bmat  \rho  \\ &\chi^{k-2} & *_2 \\ && \chi^{k-1} \emat$, then this rank condition forces $*_2$ to be unramified at $\ell$. Thus in this case, $\sigma_{\Lambda'}$ has a direct summand isomorphic to $\bmat \chi^{k-2} & * \\ & \chi^{k-1}\emat $ with $*$ unramified away from $p$. This direct summand is short crystalline at $p$ since $\sigma_{\Lambda'}$ is, and so gives rise to an element in $H^1_f(\bfQ, \bfF(-1))$. Since this group is trivial (by Proposition \ref{class groups}), this implies that  $\ov{\sigma}_{\Lambda'}$ is semisimple, contradicting our assumption. 

So we must have that $*_3$ gives a non-trivial extension of $\chi^{k-1}$ by $\rho$. Now, apply Theorem \ref{centralizer} with $\rho_1=\chi^{k-2}$ and $\rho_2=\bmat \rho & *_3 \\ & \chi^{k-1}\emat$ (note that $\rho_2$ has scalar centralizer). This gives us   \be \label{eq19} \ov{\sigma}_{\Lambda} = \bmat \chi^{k-2} &*_1 & *_2 \\ & \rho & *_3 \\ && \chi^{k-1} \emat\not\cong \chi^{k-2} \oplus \bmat \rho & *_3 \\ & \chi^{k-1}\emat\ee with $*_3$ a non-split extension of $\chi^{k-1}$ by $\rho$. It remains to show that $*_1$ gives rise to a non-trivial extension of $\rho$ by $\chi^{k-2}$. Suppose that one has $\bmat \chi^{k-2} & *_1 \\ & \rho \emat \cong \chi^{k-2} \oplus \rho$. Then $*_1$ corresponds to a coboundary in $H^1(\bfQ, \Hom(\rho, \chi^{k-2}))$. More precisely, there exists a matrix $f \in M_{n_1 \times n_2}(\bfF)$ (i.e., a map in $\Hom_{\bfF}(\rho, \chi^{k-2})$) such that $$\bmat \chi^{k-2} & *_1 \\ & \rho\emat \cong \bmat \chi^{k-2} & \chi^{k-2} f - f \rho\\ 0 & \rho\emat.$$ However, note that $$\bmat I_{n_1} & f \\ & I_{n_2} \\ && I_{n_3} \emat \bmat \chi^{k-2} & a & b \\ & \rho & c \\ && \chi^{k-1} \emat \bmat I_{n_1} & f \\ & I_{n_2} \\ && I_{n_3} \emat^{-1} = \bmat \chi^{k-2} & 0 & b+fc \\ & \rho & c \\ && \chi^{k-1} \emat.$$ So, it remains to show that $$ \bmat \chi^{k-2} & 0 &b \\ & \rho & c \\ && \chi^{k-1} \emat \cong \chi^{k-2} \oplus \bmat \rho & * \\ & \chi^{k-1}\emat.$$ Note that we have $$\bmat 0_{n_2\times n_1} & I_{n_2} \\ I_{n_1} & 0_{n_1\times n_2} \\ && I_{n_3}\emat  \bmat \chi^{k-2} & 0 &b \\ & \rho & c \\ && \chi^{k-1} \emat \bmat 0_{n_2\times n_1} & I_{n_2} \\ I_{n_1} & 0_{n_1\times n_2} \\ && I_{n_3}\emat^{-1} = \bmat \rho && c \\ & \chi^{k-2} & b \\ && \chi^{k-1} \emat.$$  By the same argument as before $b \in H^1(G_{\Sigma}, \Hom_{\bfF}(\chi^{k-1}, \chi^{k-2}))$ must again be a coboundary, i.e., $b(g) = \chi^{k-2}(g) f -f \rho(g)$ for some matrix $f$. But then again the matrix $\bmat I_{n_1} \\ & I_{n_2} & f \\ &&I_{n_3}\emat$ conjugates $\bmat \rho && c \\ & \chi^{k-2} & b \\ && \chi^{k-1} \emat$ to $\bmat \rho && c \\ & \chi^{k-2} & 0 \\ && \chi^{k-1} \emat \cong \chi^{k-2} \oplus \bmat \rho &*_3 \\ & \chi^{k-1} \emat$, which leads to a contradiction to \eqref{eq19}.
The fact that $\ov{\sigma}_{\Lambda}$ is semi-abelian follows again from the minimality of $\sigma$.
\end{proof} 

\begin{cor} \label{corlat} Let $\sigma, \tau: G_{\Sigma} \to \GL_4(E)$ be two absolutely irreducible short crystalline at $p$  Galois representations such that their restriction to ${I_{\ell}}$  is isomorphic to $\exp(t_p N_1)$ for all $\ell \in \Sigma \setminus \{p\}$.      Suppose that $\ov{\sigma}^{\rm ss} =\ov{\tau}^{\rm ss}= \chi^{k-2} \oplus \rho \oplus \chi^{k-1}$ with $\rho$ absolutely irreducible. Suppose that $\dim_{\bfF}H^1_{f}(\bfQ, \rho(1-k))\leq 1$. 
Then there exists  a $G_{\Sigma}$-stable lattice $\Lambda$ in the space of $\sigma$ and  a $G_{\Sigma}$-stable lattice $\Lambda'$ in the space of $\tau$ such that $$\ov{\sigma}_{\Lambda} = \bmat \chi^{k-2} & a & b \\ & \rho & c \\ && \chi^{k-1} \emat \quad \textup{and} \quad  \ov{\tau}_{\Lambda'}=\bmat \chi^{k-2} & a & b' \\ & \rho & c \\ && \chi^{k-1} \emat,$$ both semi-abelian, with $a,c$ both non-trivial elements of $H^1_{f}(\bfQ, \rho(1-k))$. \end{cor}

\begin{rem} Note that in Corollary \ref{corlat} we only assume an upper bound on the dimension of $H^1_f(\bfQ, \rho(1-k))$, but in fact as argued  at the beginning  of the proof of the corollary  our assumptions imply that we always have $\dim_{\bfF}H^1_f(\bfQ, \rho(1-k))\geq 1$. \end{rem}
\begin{proof} By Corollary \ref{latt} we can find lattices, so that both  reductions  have the above shape, are short crystalline at $p$ and semi-abelian. By Corollary \ref{splitting at ell} we know that the extensions induced by the entries $a$ and $c$ lie in $H^1_{f}(\bfQ, \rho(1-k))$. Since the latter group is one-dimensional  we can conjugate one of them to ensure that both representations have the same $a$-entries and $c$-entries. \end{proof}

\section{Uniqueness of iterated residual extensions} \label{Uniqueness of iterated residual extensions}
 In this section we assume that $\rho$ is the mod $\varpi$ reduction of $\rho_f: G_{\Sigma} \to \GL_2(E)$, the Galois representation attached to a newform $f \in S_{2k-2}(N)$, 
where $N$ squarefree,  $p \nmid N$, $k$ is even and $\Sigma=\{\ell \mid N\}\cup \{p\}$. We also assume that $\rho$ is absolutely irreducible and ramified at every $\ell \mid N$. By Proposition \ref{invariants}  (see Corollary \ref{l not equal to p}) \ we then have that $H^1_{f, \rm min}(\bfQ, \rho(1-k)) = H^1_{f, \rm ur}(\bfQ, \rho(1-k))$ which we will simply abbreviate to $H^1_f(\bfQ, \rho(1-k))$. 
 We assume in this section that $$ \dim_{\bfF}H^1_f(\bfQ, \rho(1-k)) =1.$$ 

The goal of this section is to prove the following proposition. 

\begin{prop} \label{uni1} Let $\ov{\sigma}, \ov{\tau}: G_{\Sigma} \to \GL_4(\bfF)$ be two Galois representations of the form $$\ov{\sigma}= \bmat \chi^{k-2} & a & b \\ & \rho & c \\ && \chi^{k-1} \emat, \quad \ov{\tau}=\bmat \chi^{k-2} & a' & b' \\ & \rho & c' \\ && \chi^{k-1} \emat$$ with the extensions $\bmat \chi^{k-2} & a\\ & \rho \emat$, $\bmat \chi^{k-2} & a'\\ & \rho \emat$, $\bmat \rho & c\\ & \chi^{k-1} \emat$ and $\bmat \rho & c'\\ & \chi^{k-1} \emat$ all giving rise to non-zero elements of $H^1_{f}(\bfQ, \rho(1-k))=\bfF$.
  Suppose also that both $\ov{\sigma}$ and $\ov{\tau}$ are short crystalline and semi-abelian. 
Then $\ov{\sigma} \cong \ov{\tau}$. \end{prop}

\begin{proof} Denote the cohomology classes corresponding to the extensions in the statement of the Proposition by $\phi, \phi', \psi, \psi'$ respectively. 
Note that they indeed give rise to elements in $H^1_f$ 
by Corollary \ref{splitting at ell}.
Then there exist $\alpha, \gamma \in \bfF^{\times}$ such that  $\phi'=\alpha \phi$ and $\psi'=\gamma \psi$. The functions $a,a', c,c'$ are given by $a(g)=\phi(g)\rho(g)$, $a'(g)=\phi'(g)\rho(g)$, $c(g)=\psi(g)\chi^{k-1}(g)$ and $c'(g)=\psi'(g) \chi^{k-1}(g)$, so in particular $a'=\alpha a$ and $c'=\gamma c$. Hence 
$$\bmat 1 \\ & \alpha I_2 \\ && \alpha \gamma \emat \ov{\tau} \bmat 1 \\ & \alpha I_2 \\ && \alpha \gamma \emat^{-1} = \bmat \chi^{k-2} & a & \alpha^{1}\gamma^{-1} b \\ & \rho & c \\ &&\chi^{k-1}\emat.$$ Thus we may assume without loss of generality that $a'=a$ and $c'=c$. We will express $\ov{\sigma}$ as an  iterated extension and apply Lemma \ref{Sel_funct} to show that such extensions are a torsor under $H^1_f(\bfQ, \bfF(-1))=0$ from which the uniqueness of such $\ov{\sigma}$ will follow.

Consider $\ov{\sigma}$ as in the statement of the Proposition. Its subrepresentation corresponding to the $3\times 3$ upper-left block  gives rise to the following exact sequence of $G_{\Sigma}$-representations: 
\be \label{ses} 0 \to \Hom_{\bfF}(\chi^{k-1}, \chi^{k-2}) \to  \Hom_{\bfF}(\chi^{k-1}, \bmat\chi^{k-2} & a \\ & \rho \emat) \to \Hom_{\bfF}(\chi^{k-1}, \rho) \to 0,\ee and after restricting to the decomposition group at $v$, also to a corresponding short exact sequence of $G_{\bfQ_v}$-representations.

We first consider the case $v=p$: Write $\mM$ for the category of filtered Dieudonne modules as defined in \cite{ClozelHarrisTaylor08} Section 2.4.1. 
Since $\mM$ is closed under subquotients, there exist objects $M_1, M_2, M_3, M$ of $\mM$ corresponding to the restrictions to $G_{\bfQ_p}$  of $\chi^{k-2}$, $\rho$, $\chi^{k-1}$ and the subrepresentation $\bmat\chi^{k-2} & a \\ & \rho \emat: G_{\Sigma} \to \GL_3(\bfF)$ respectively. 
Since $\Hom(\chi^{k-1}, M) = M(1-k)$ we see that all the objects in \eqref{ses} are in the essential image of $\bfG$. Hence we can apply Lemma 5.3 in \cite{BergerKlosin13} to conclude that it gives rise to an exact sequence of Selmer groups   \begin{multline} \label{les1} H^0(\bfQ_p,\Hom_{\bfF}(\chi^{k-1}, \rho)) \to H^1_f(\bfQ_p, \Hom_{\bfF}(\chi^{k-1}, \chi^{k-2})) \to H^1_f(\bfQ_p, \Hom_{\bfF}(\chi^{k-1}, \bmat \chi^{k-2} & a \\0&\rho \emat ))\\ \to H^1_f(\bfQ_p,\Hom_{\bfF}(\chi^{k-1}, \rho)) \to 0.\end{multline}
It follows from \eqref{les1} that (in the terminology of \cite{Weston00}, p.4) the finite-singular structures on the first and last non-zero term in the $G_{\bfQ_p}$ sequence \eqref{ses} are induced from the middle term.  

 Let us now show that this is also true for all primes $v \neq p$ when we take $H^1_f$ to be the unramified structure defined in section \ref{Setup}. 
 So, suppose that $v \mid N$. It suffices to show that the sequence \be \label{unr1} \begin{split}
H^0(\bfQ_{v},\Hom_{\bfF}(\chi^{k-1}, \rho)) \to  H^1_{\rm ur}(\bfQ_v, \Hom_{\bfF}(\chi^{k-1},\chi^{k-2})) \to H^1_{\rm ur}(\bfQ_v, \Hom_{\bfF}(\chi^{k-1},\bmat \chi^{k-2} & a \\0&\rho \emat)) \\
\to H^1_{\rm ur}(\bfQ_v, \Hom_{\bfF}(\chi^{k-1},\rho)) \to 0\end{split}\ee
with $H^1_{\rm ur}(\bfQ_v, M) := \ker (H^1(\bfQ_v, M) \to H^1(I_v, M))$ is exact. 
 Let us rewrite \eqref{ses} as \be \label{ses2} 0 \to \bfF(-1)\to   \bmat\chi^{-1} & a' \\ & \rho(1-k) \emat) \to  \rho(1-k) \to 0\ee 
 and call the extension in the middle $\mE$. 

By  Corollary \ref{splitting at ell} we have
  $\mE^{I_v} = \bfF(-1)^{I_v} \oplus \rho(1-k)^{I_v},$ so in particular the sequence $$0 \to  \bfF(-1)^{I_v}  \to \mE^{I_v} \to \rho(1-k)^{I_v} \to 0$$ is exact. Since every module in that sequence has an action of $G:=\Gal(\bfQ_v^{\rm ur}/\bfQ_v)$, we get a long exact sequence in cohomology of that group (and note that $(M^{I_v})^{G} = M^{G_v}$, where $G_v=G_{\bfQ_v}$):
\begin{multline} 0 \to\bfF(-1)^{G_v} \to \mE^{G_v} \to \rho(1-k)^{G_v} \to H^1(\bfQ_v^{\rm ur}/\bfQ_v, \bfF(-1)^{I_v}) \to \\
 H^1(\bfQ_v^{\rm ur}/\bfQ_v, \mE^{I_v}) \to  H^1(\bfQ_v^{\rm ur}/\bfQ_v, \rho(1-k)^{I_v}) \to  H^2(\bfQ_v^{\rm ur}/\bfQ_v, \bfF(-1)^{I_v}) \end{multline}
Note that the last group is zero since $G$ has cohomological dimension one. For all three modules $M$ one has $H^1(\bfQ_v^{\rm ur}/\bfQ_v, M^{I_v})=H^1_{\rm ur}(\bfQ_v, M)$ by Lemma 3.2 (i) in \cite{Rubin00}. 

Since we now showed that the finite-singular structures on the first and last non-zero term in the sequence \ref{ses} are induced from the middle term for all $G_{\bfQ_{v}}$, Lemma \ref{Sel_funct} tells us that we have an exact sequence 
\be \label{selgps}
 H^1_{f}(\bfQ, \Hom_{\bfF}(\chi^{k-1},\chi^{k-2})) \to H^1_{f}(\bfQ, \Hom_{\bfF}(\chi^{k-1},\bmat \chi^{k-2} & a \\0&\rho \emat)) \to H^1_{f}(\bfQ, \Hom_{\bfF}(\chi^{k-1},\rho)).\ee
Since $H^1_{f}(\bfQ, \Hom_{\bfF}(\chi^{k-1},\chi^{k-2})) = H^1_{f}(\bfQ, \bfF(-1))=0$ by Proposition \ref{class groups},
we see that the middle Selmer group in \eqref{selgps} injects into the last one. Hence 
the cohomology class in  $H^1_{f}(\bfQ, \Hom_{\bfF}(\chi^{k-1}, \rho))$ corresponding to the quotient representation $\bmat \rho & c \\ 0 & \chi^{k-1}\emat$  of $\ov{\sigma}$ determines a short crystalline extension of $\chi^{k-1}$ by $\bmat \chi^{k-2} & a \\ 0 & \rho\emat$ uniquely.  Since both $\ov{\sigma}$ and $\ov{\tau}$ are such extensions, we get $\ov{\sigma} \cong \ov{\tau}$. \end{proof}

\section{The ideals of reducibility} \label{The ideals of reducibility}
Let
$\rho_f : G_{\Sigma} \to \GL_2(E)$ be the Galois representation attached to a newform $f \in S_{2k-2}(N)$ for $N$ a square-free integer with $p \nmid N$ and $\Sigma=\{ \ell \mid N \} \cup \{p\}$. Suppose also that the residual representation $\rho:=\ov{\rho}_f$ is absolutely irreducible and ramified at every prime $\ell \mid N$. Then $\rho|_{I_{\ell}} \cong \bmat 1&*\\ &1\emat$ (see \eqref{min for rho 2}). Let $$\ov{\sigma}=\bmat \chi^{k-2} & a&b \\ & \rho &c \\ &&\chi^{k-1}\emat:G_{\Sigma} \to \GL_4(\bfF)$$ be short crystalline at $p$ and semi-abelian with $a,c$ non-trivial classes in $H^1_f(\bfQ, \rho(1-k))$.

The  universal deformation  $\sigma^{\rm red}$ (resp. its trace) gives rise to an $R^{\rm red}$-algebra morphism (which we denote by the same letter) $\sigma^{\rm red}: R^{\rm red}[G_{\Sigma}] \to M_4(R^{\rm red})$ (resp. $T=\tr \sigma^{\rm red}: R^{\rm red}[G_{\Sigma}]\to R^{\rm red}$). We define $\ker \sigma^{\rm red}$ (resp. $\ker T$) as in \cite{BellaicheChenevierbook}, section 1.2.4. 

Since $R$, and hence also $R^{\rm red}$, is Noetherian, both rings have a finite number of minimal primes. Then the total ring of fractions ${\rm Tot}(R^{\rm red})$ of $R^{\rm red}$ is a finite product of fields. In fact we have an injection $$R^{\rm red} \hookrightarrow {\rm Tot}(R^{\rm red})= \prod_{\mP}R^{\rm red}_{\mP},$$ where $\mP$ runs over minimal primes of $R^{\rm red}$ and the localization $R^{\rm red}_{\mP}$ is a field (cf. \cite{BellaicheChenevierbook}, Proposition 1.3.11).

To ease notation in this section we will write $G$ for $G_{\Sigma}$. 
Let us note that both $R^{\rm red}[G]/\ker T$ and $R^{\rm red}[G]/\ker \sigma^{\rm red}$ are Cayley-Hamilton quotients of $R^{\rm red}[G]$ in the sense of \cite{BellaicheChenevierbook}, section 1 and we have a canonical $R^{\rm red}$-algebra map $\varphi : R^{\rm red}[G]/\ker \sigma^{\rm red} \twoheadrightarrow R^{\rm red}[G]/\ker T$. 
\begin{thm} \label{1.4.4} One has the following:
\begin{itemize}
\item[(i)] Let $? \in \{T, \sigma\}$ and $i,j \in \{k-1, \rho, k-2\}$. Let $S^?$ be $R^{\rm red}[G]/\ker T$ (when $?=T$) or $R^{\rm red}[G]/\ker \sigma^{\rm red}$ (when $?=\sigma$). Then there  are  data $\mE$ of idempotents for which $S^?$ is a GMA (Generalized Matrix Algebra) and there exist $R^{\rm red}$-submodules $\mA^?_{i,j}$ of $S^?$ which satisfy $$\mA^?_{i,j}\mA^?_{j,k} \subset \mA^?_{i,k}, \quad T: \mA^?_{i,i} \xrightarrow{\sim} R^{\rm red}, \quad T(\mA^?_{i,j}\mA^?_{j,i}) \subset \fm_{R^{\rm red}},$$ and $$S^? \cong {\rm GMA}(\mA^?):=\bmat \mA^?_{k-2, k-2} &M_{1, 2}(\mA^?_{k-2, \rho}) & \mA^?_{k-2, k-1}\\ M_{2,1}(\mA^?_{\rho, k-2}) & M_{2,2}(\mA^?_{\rho, \rho}) & M_{2,1}(\mA^?_{\rho, k-1}) \\ \mA^?_{k-1, k-2} & M_{1,2}(\mA^?_{k-1, \rho}) & \mA^?_{k-1, k-1}\emat.$$ The isomorphism $S^? \cong {\rm GMA}(\mA^?)$ (and also $S^?\cong {\rm GMA}(A^?)$ below) is an $R^{\rm red}$-algebra morphism.
\item[(ii)] For $\mE$ as in (i) one has  $$S^? \cong {\rm GMA}(A^?):=\bmat R^{\rm red} &M_{1, 2}(A^?_{k-2, \rho}) & A^?_{k-2, k-1}\\ M_{2,1}(A^?_{\rho, k-2}) & M_{2,2}(R^{\rm red}) & M_{2,1}(A^?_{\rho, k-1}) \\ A^?_{k-1, k-2} & M_{1,2}(A^?_{k-1, \rho}) & R^{\rm red} \emat \subset M_4({\rm Tot}(R^{\rm red})),$$  where $A^?_{i,j}$ are fractional ideals of $R^{\rm red}$ (and $A^{\sigma}_{i,j}$ are ideals of $R^{\rm red}$) that satisfy the \emph{Chasles relations} $$A^?_{i,j}A^?_{j,i} \subset A^?_{i,k}, \quad A^?_{i,i} = R^{\rm red}, \quad A^?_{i,j}A^?_{j,i} \subset \fm_{R^{\rm red}}.$$
\item[(iii)] The data $\mE$ can be adapted so that the $R^{\rm red}$-algebra map $\varphi': {\rm GMA}(\mA^{\sigma}) \twoheadrightarrow {\rm GMA}(\mA^T)$ induced from $\varphi$ has the property that $\varphi'(\mA^{\sigma}_{i,j})=\mA^T_{i,j}$. 
\end{itemize}
 \end{thm}
\begin{proof} The statements (i), (ii) for $?=T$ and (i) for $?=\sigma$ are a direct consequence of Theorem 1.4.4 in \cite{BellaicheChenevierbook}. To show that we can also obtain (ii) for $?=\sigma$ we argue as in \cite{BergerKlosin13}, Proposition 2.8 which uses Lemma 1.3.7 of \cite{BellaicheChenevierbook}. Finally, we proceed as in the proof of Lemma 2.5 in \cite{BergerKlosin13} to show that $\mE$ can be adapted to satisfy (iii). \end{proof}

\begin{rem} \label{adaptness11}  The data $\mE=\{e_i, \psi_i\mid i \in \{\chi^{k-2}, \rho, \chi^{k-1}\}\}$ in Theorem \ref{1.4.4} can be (and will be) chosen so that for $i \in \{\chi^{k-2}, \rho, \chi^{k-1}\}$ the maps $\psi_i: e_i S^{\sigma} e_i \xrightarrow{\sim} M_{d_i, d_i}(R^{\rm red})$ satisfy $\psi_i\otimes \bfF \cong i$. Here $d_i=\dim i$. From now on fix $\mE$. \end{rem}

Using Theorem \ref{1.4.4} we get the following commutative diagram of $R^{\rm red}$-algebras:
$$\xymatrix{{\rm GMA}(\mA^{\sigma}) \ar[d]^{\varphi'}&R^{\rm red}[G_{\Sigma}]/\ker \sigma^{\rm red} \ar[l]^{\iota_1}_{\sim} \ar[r]^{\iota_2}_{\sim} \ar[d]^{\varphi} & {\rm GMA}(A^{\sigma})\\ {\rm GMA}(\mA^{T}) & R^{\rm red}[G_{\Sigma}]/\ker T  \ar[l]^{\iota_3}_{\sim} \ar[r]^{\iota_4}_{\sim} & {\rm GMA}(A^T)}$$ where the maps $\iota_i$ are the ones given by Theorem \ref{1.4.4}. Using Lemma 1.3.8 (resp. Proposition 1.3.12) of \cite{BellaicheChenevierbook} we get that the composite $\iota_2 \circ \iota_1^{-1}$ (resp. $\iota_4 \circ \iota_3^{-1}$) give rise to isomorphisms of $R^{\rm red}$-modules $f_{i,j}^{\sigma}: \mA^{\sigma}_{i,j} \xrightarrow{\sim} A^{\sigma}_{i,j}$ (resp. $f_{i,j}^{T}: \mA^{T}_{i,j} \xrightarrow{\sim} A^{T}_{i,j}$). We define  $R^{\rm red}$-module surjections $\phi_{i,j}$  to make the following diagram commute \be \label{diag phi} \xymatrix{\mA^{\sigma}_{i,j} \ar[r]^{f_{i,j}^{\sigma}} \ar[d]_{\varphi'} & A^{\sigma}_{i,j} \ar[d]^{\phi_{i,j}} \\ \mA^T_{i,j} \ar[r]_{f^T_{i,j}} & A^T_{i,j}}\ee

\begin{definition}[\cite{BellaicheChenevierbook} Definition 1.5.2] Let $\mP=(\mP_1, \dots, \mP_s)$ be a partition of the set $\mI  = \{\chi^{k-2}, \rho, \chi^{k-1}\}$. The ideal of reducibility $I^{\mP}$ (associated with partition $\mP$) is the smallest ideal $I$ of $R^{\rm red}$ with the the property that
 there exist pseudocharacters $T_1, \dots, T_s: R^{\rm red}[G]/I R^{\rm red}[G]  \to R^{\rm red}/I$ such that \begin{itemize}
\item[(i)] $T \otimes R^{\rm red}/I = \sum_{l=1}^s T_l$,
\item[(ii)] for each $l \in \{1, \dots, s\}$, $T_l \otimes \bfF =\sum_{\rho' \in \mP_l} \tr \rho'$. \end{itemize} \end{definition}

To shorten notation we will sometimes write $k-1, k-2, \rho$ for the elements of $\mI$ instead of $\chi^{k-1}, \chi^{k-2}, \rho$. 
\begin{prop} [Proposition 1.5.1 in \cite{BellaicheChenevierbook}] \label{CH1} For every partition $\mP$ the corresponding ideal of reducibility $I^{\mP}$ exists. Furthermore, let $S$ be any Cayley-Hamilton quotient of $(R^{\rm red}, T)$ (we will only use $(R^{\rm red}, T)$ and $(R^{\rm red}, \sigma^{\rm red})$) and choose   data of idempotents $\mE$ as in Theorem \ref{1.4.4} so that $S\cong {\rm GMA}(\mA^?)$. Then $I^{\mP}$ is given by the following formula (whose sides do not depend on the choice of $S$ or $\mE$) $$I^{\mP}=\sum_{\substack{(i,j)\\ \textup{$i,j$ not in the same $\mP_l$}}} T(\mA^?_{i,j} \mA^?_{j,i}).$$  \end{prop}

\begin{cor} For $?\in \{T, \sigma^{\rm red}\}$ let $A^?_{i,j}$ be defined as in Theorem \ref{1.4.4} (ii). Then one has $$I^{\mP} = \sum_{\substack{(i,j)\\ \textup{$i,j$ not in the same $\mP_l$}}} A^?_{i,j} A^?_{j,i}.$$ \end{cor}
\begin{proof} Multiplication between elements of $\alpha \in \mA^?_{i,j}$ and $\beta \in \mA^?_{j,k}$ corresponds to correct matrix multiplication, i.e., one puts 
$\alpha$ in the $(i,j)$th spot (which may be a block) of a matrix and $\beta$ in the $(j,k)$th spot and completes both matrices by putting zeros elsewhere. Then $\alpha \beta$ is the $(i,k)$th spot in the matrix obtained as a product of the matrices above.  The corollary follows from  the commutativity of the following diagram $$\xymatrix{\mA^?_{i,j} \mA^?_{j,i} \ar@{^{(}->}[r] \ar[d]_{f_{i,j}^?\otimes f_{j,i}^?} & \mA_{i,i} \ar[d]_{\sim}^T \\ A^?_{i,j}A^?_{j,i} \ar@{^{(}->}[r] & R^{\rm red}}$$   since $f^?_{i,i}=T:\mA_{i,i} \to R^{\rm red}$.  \end{proof}

In our situation we have 4 possible partitions of $\mP$ and the following 4 corresponding ideals of reducibility. 
\be \begin{split}
\mP = \{\chi^{k-2}, \rho\} \cup \{\chi^{k-1}\} \implies  I^{\mP}=I^{k-1} =& A_{k-1, k-2}A_{k-2,k-1}+ A_{\rho, k-1}A_{k-1, \rho}\\
\mP = \{\chi^{k-1}, \rho\} \cup \{\chi^{k-2}\} \implies  I^{\mP} =I^{k-2}=& A_{k-1, k-2}A_{k-2,k-1}+ A_{\rho, k-2}A_{k-2, \rho}\\
\mP = \{\chi^{k-2}, \chi^{k-1}\} \cup \{\rho\} \implies  I^{\mP} =I^{\rho}=& A_{k-2, \rho}A_{\rho,k-2}+ A_{k-1, \rho}A_{\rho, k-1}\\
\mP = \{\chi^{k-2}\}\cup\{ \rho\} \cup \{\chi^{k-1}\} \implies  I^{\mP}=I^{\rm tot} =& A_{k-1, k-2}A_{k-2,k-1}+ A_{\rho, k-1}A_{k-1, \rho}\\
+& A_{\rho, k-2}A_{k-2, \rho}
\end{split}\ee

 We have the following analogue of \cite{BellaicheChenevierbook}, Lemma 9.3.1:
\begin{thm} \label{4}  One has \begin{itemize}
\item[(i)] All the ideals of reducibility coincide with $I^{\rm tot}$,
\item[(ii)] $I^{\rm tot} = A^T_{k-2, \rho}A^T_{\rho,k-2}$,
\item[(iii)] If, in addition, $\dim_{\bfF}H^1_f(\bfQ, \rho(2-k))=1$ and $p \nmid 1+w_{f, \ell}\ell$ for all $\ell \mid N$, then $I^{\rm tot}$ is a principal ideal of $R^{\rm red}$.
\end{itemize}
\end{thm}
\begin{proof} We will prove Theorem \ref{4} by a sequence of Lemmas some pertaining to $A_{i,j}^T$ and some to $A_{i,j}^{\sigma}$. 
\begin{lemma}\label{8.2.16} One has 
\be \label{invo}A^T_{k-2,\rho} A^T_{\rho,k-2} = A^T_{k-1, \rho}A^T_{\rho, k-1}.\ee \end{lemma}
\begin{proof} This is proved exactly as
 Lemma 8.2.16 of \cite{BellaicheChenevierbook} using the antiinvolution $\tau$. \end{proof}
We will need the following analogue of \cite{BellaicheChenevierbook}, Lemma 8.3.1: 
\begin{lemma} \label{char} We have 
$$A^{\sigma}_{k-1,k-2} = A^{\sigma}_{k-1, \rho} A^{\sigma}_{\rho, k-2}.$$ 
\end{lemma}
The strategy to prove Lemma \ref{char} follows that of \cite{BellaicheChenevierbook}, Lemma 8.3.1, but instead of using $\Ext_T$ we use Lemma \ref{injection1} below.
 For $\rho_i,\rho_j \in \{\chi^{k-1}, \chi^{k-2}, \rho\}$, $\rho_i\not\cong \rho_j$, set  $A'_{i,j}:= A^{\sigma}_{i, l}A^{\sigma}_{l, j}$, where $\rho_l \in \{\chi^{k-1}, \chi^{k-2}, \rho\}$, $\rho_i \not\cong \rho_l \not\cong \rho_j$. 
\begin{lemma} \label{injection1} There is an injection $$\Hom_{R^{\rm red}}(A^{\sigma}_{i,j}/A'_{i,j}, \bfF) \hookrightarrow H^1_{\Sigma}(\bfQ, \Hom(\rho_j,\rho_i)).$$  If either (i) or (ii) hold, where \begin{itemize}
\item[(i)] $\rho_i=\chi^{k-2}$ and $\rho_j=\chi^{k-1}$,
\item[(ii)] $\rho_i=\chi^{k-1}$ and $\rho_j=\chi^{k-2}$,
\end{itemize} then the image of the injection is contained in $H^1_f$. \end{lemma}
\begin{proof} For simplicity in this proof only we write $R$ for $R^{\rm red}$. By \cite{BellaicheChenevierbook}, Theorem 1.5.5, taking $J=\fm$ there is an injection
$$\iota_{i,j}: \Hom_R(A^{\sigma}_{i,j}/A'_{i,j}, \bfF) \hookrightarrow  \Ext^1_{R[G]/\fm R[G]}(\rho_j, \rho_i).$$ We have $$\fm R[G]=\ker (R[G]\to \bfF[G]) = (\fm R)[G],$$ hence $R[G]/\fm R[G] = (R/\fm R)[G]= \bfF[G].$ Thus we get an injection 
 $$\iota_{i,j}: \Hom_R(A^{\sigma}_{i,j}/A'_{i,j}, \bfF) \hookrightarrow H^1(\bfQ, \Hom(\rho_j,\rho_i)).$$ It remains to prove that the image is contained in the correct Selmer group. By \cite{BellaicheChenevierbook}, Theorem 1.5.6(1) the image consists precisely of the $S/\fm S$-extensions of $\rho_i$ by $\rho_j$, where $$S:= (R[G])/(\ker \sigma)(R[G]).$$ 
 By \cite{BellaicheChenevierbook}, Theorem 1.5.6(2) any $S/\fm S$-extension is a quotient of $M_j/\fm M_j \oplus \rho_i$, where  $M_j=SE_j$ and $E_j$ are defined as in \cite{BellaicheChenevierbook}, p. 21. 
Since $p > 2k-2$, the representations $\chi^{k-2}$, $\chi^{k-1}$ and $\rho$ are short crystalline.
Since the category of such representations is closed under taking subobjects, quotients and finite direct sums it suffices therefore to prove that $M_j$ is short crystalline.
By \cite{BellaicheChenevierbook}, section 1.5.4, one has $S=M_j \oplus S(1-E_j)$, hence in particular $M_j$ is an $S$-submodule, and hence also an $\Oo[G]$-submodule of $S$.
One has $M_j \subset M_j \otimes {\rm Tot}( R^{\rm red})$ and the latter Galois module is isomorphic to the representation $({\rm Tot}(R^{\rm red})^4, \sigma^{\rm red})$, which is short  crystalline. Thus $M_j$ is short crystalline.

Let us now check that if $\rho_i=\chi^{k-2}$ and $\rho_j=\chi^{k-1}$ (or vice versa), then the extensions in the image of $\iota_{i,j}$ are unramified away from $p$. Let $\ell \mid N$ be a prime. By Remark \ref{adaptness11}  we can conjugate $\sigma^{\rm red}$ so that it is adapted to the data of idempotents $\mE$. Abusing notation we will in this proof denote this conjugate still by $\sigma^{\rm red}$. 
This implies in particular that $e_{\rho} (\sigma\otimes \bfF) e_{\rho} \cong  \rho$  and that \be\label{adaptness22}\sigma^{\rm red}(S^{\sigma})=\bmat R  &M_{1, 2}(A^{\sigma}_{k-2, \rho}) & A^{\sigma}_{k-2, k-1}\\ M_{2,1}(A^{\sigma}_{\rho, k-2}) & M_{2,2}(R) & M_{2,1}(A^{\sigma}_{\rho, k-1}) \\ A^{\sigma}_{k-1, k-2} & M_{1,2}(A^{\sigma}_{k-1, \rho}) &R\emat.\ee  Let $X\in I_{\ell}$ be such that $\sigma^{\rm red}(X)$ topologically generates $\sigma^{\rm red}(I_{\ell})$. Write $\alpha_{k,l}$ for the $(k,l)$-entry of $Y:=\sigma^{\rm red}(X)-I_4$. Since $\rho$ is ramified at $\ell$ at least one of the  entries $\alpha_{2,2}$, $\alpha_{2,3}$, $\alpha_{3,2}$, $\alpha_{3,3}$  lies in $R^{\times}$. To fix attention let us assume that $\alpha_{2,3}\in R^{\times}$. The proof in the other three cases is identical. 

The construction of the extensions in the image of $\iota_{i,j}$ is given in \cite{BellaicheChenevierbook}, p. 37. 
Note that while there are in general three choices for $i$ (corresponding to $\rho_i=\chi^{k-2}, \chi^{k-1}$ or $\rho$), and the same holds for $j$, there are four choices for $k$ and for $l$. 

If $\rho_{i}=\rho_{j}$, let $a_{ii}: S/\fm S \to e_{\rho_{i}} (S/\fm S) e_{\rho_{i}}$ be the canonical projection defined in Lemma 1.5.4 of \cite{BellaicheChenevierbook}, where $e_{\rho_i}$ is the idempotent corresponding to $\rho_i$. (For notation and terminology see \cite{BellaicheChenevierbook}.) We note that $a_{ii}$ composed with the isomorphism $ e_{\rho_i} (S/\fm S) e_{\rho_i} \cong \bfF$ gives a representation isomorphic to $\rho_i$. If $\rho_i \neq \rho_j$, then the projection   $a_{ij}: S/\fm S \to e_{\rho_i}(S/\fm S) e_{\rho_{j}}$ is defined in an analogous way  and is the mod $\fm$-reduction of the canonical projection $\tilde{a}_{ij}: S \to e_{\rho_i} S e_{\rho_{j}}$ defined again in the same way. The codomain of $\tilde{a}_{ij}$ is isomorphic to $\mA^{\sigma}_{ij}$ (via, say, a map $\phi$) which then is isomorphic to $A^{\sigma}_{ij}$ via the map $f_{i,j}^{\sigma}$.  The composite $f^{\sigma}_{i,j} \circ \phi \circ \tilde{a}_{ij}$ equals $\epsilon_i \sigma^{\rm red} \epsilon_j$, where $$\epsilon_i = \begin{cases} \diag(1,0,0,0) & \textup{if $\rho_i=\chi^{k-2}$}\\ \diag(0,1,1,0) & \textup{if $\rho_i=\rho$}\\ \diag(0,0,0,1) & \textup{if $\rho_i=\chi^{k-1}$}.\end{cases}$$
Here we are only interested in the case when $\rho_i=\chi^{k-1}$ and $\rho_j=\chi^{k-2}$ or vice versa (but the proof in the second case is identical to the first, so we omit it). 

Let $f \in \Hom(A^{\sigma}_{\chi,1}/A'_{\chi,1}, \bfF)$.  Recall that we have $A'_{\chi, 1} = A^{\sigma}_{\chi, \rho}A^{\sigma}_{\rho,1}$. Since $\chi$ and $1$ are unramified to show that the extension $$\bmat a_{\chi\chi} & f \circ a_{\chi 1}\\ & a_{11}\emat$$ is unramified at $\ell$ it suffices to show that $a_{\chi 1}(I_{\ell}) \subset A'_{\chi,1}$. Note that $a_{\chi 1}(X)=\alpha_{4,1}$.  
Since $Y$ is a conjugate of $\bmat 0&0&0&0\\ 0&0&1&0\\ 0&0&0&0\\ 0&0&0&0\emat$ and $A^{\sigma}_{k-1, \rho}\subset R$ (by Theorem \ref{1.4.4} (ii)) while $\alpha_{2,3} \in R^{\times}$ we see that the fourth row of $Y$ is a scalar multiple of the second row and that this scalar $\alpha\in {\rm Tot}(R)$ in fact lies in $R$.   If $\alpha=0$, then $a_{\chi 1}(X)=\alpha_{4,1}=0$ and we are done, otherwise let $s$ be the largest integer such that $\alpha \in \fm^s$. Then $\alpha_{4,1}=\alpha \alpha_{2,1}$. If $\alpha_{2,1}=0$, we are done, otherwise let $r$ be the largest integer such that $\alpha_{2,1} \in \fm^r$. Then $\alpha_{4,1}=\alpha\alpha_{2,1} \in \fm^{r+s}$. Finally $\alpha_{4,3} = \alpha\alpha_{2,3} \in \fm^s- \fm^{s+1}$ since $\alpha_{2,3}$ is a unit. Comparing the matrix 
$$Y=\bmat *&*&*&*\\ \alpha_{2,1} &*&\alpha_{2.,3}&*\\ *&*&*&*\\ \alpha\alpha_{2,1}&*&\alpha\alpha_{2,3}&*\emat$$ with \eqref{adaptness22}     we see that by the definition of $r$ we get that $A^{\sigma}_{\rho,1} \supset \fm^r$ (since $A^{\sigma}_{\rho,1}$ is the ideal of $R$ generated by the entries of $\sigma$ falling in the corresponding spots and $\alpha_{2,1} \in \fm^r - \fm^{r+1}$). Similarly by the definition of $s$ and the fact that $\alpha_{4,3}  \in \fm^s- \fm^{s+1}$ we get that $A^{\sigma}_{\chi, \rho} \supset \fm^s$ and so, $A'_{\chi, 1} = A^{\sigma}_{\chi, \rho}A^{\sigma}_{\rho,1}\supset \fm^{r+s}$. Thus, since $\alpha_{4,1} \in \fm^{r+s}$ we get $\alpha_{4,1} \in A'_{\chi,1}$. 
 This completes the proof of the lemma. 
\end{proof}

\begin{proof}[Proof of Lemma \ref{char}] We have by Proposition \ref{class groups} that $H^1_{f}(\bfQ, \Hom(\chi^{k-2},\chi^{k-1}))=0$. Thus by Lemma \ref{injection1} and Nakayama's Lemma we get that $A^{\sigma}_{k-1, k-2}= A'_{k-1, k-2}$.  \end{proof}
\begin{lemma} \label{5.2.1} We have 
$$A^{T}_{k-1,k-2} = A^{T}_{k-1, \rho} A^{T}_{\rho, k-2}.$$ 
\end{lemma}

\begin{proof} This follows by applying $\phi_{k-1, k-2}$ to both sides of the equality in Lemma \ref{char} and noting that $\phi_{k-1, k-2}$ preserves multiplication after identifying the elements $\alpha \in A^{\sigma}_{k-1, \rho}$ and $\beta \in A^{\sigma}_{\rho, k-2}$ as matrices with non-zero entries in the correct spots, i.e., that one has $\phi_{k-1, k-2}(\alpha \beta) = \phi_{k-1, \rho}(\alpha) \phi_{\rho, k-2}(\beta)$.  \end{proof} 

We can now finish the proof of parts (i) and (ii) of Theorem \ref{4}. First we
apply the Chasles relation to see that  one has $A^T_{k-2, k-1} A^T_{k-1, \rho} \subset A^T_{k-2, \rho}$ (cf. Theorem \ref{1.4.4}). So, using Lemma \ref{char}  one gets $A^T_{k-2, k-1} A^T_{k-1, k-2} \subset A^T_{k-2, \rho} A^T_{\rho, k-2}$. This gives $$I^{\rm tot} = A^T_{k-2, k-1} A^T_{k-1, k-2} +  A^T_{k-2, \rho}A^T_{\rho,k-2} + A^T_{k-1, \rho}A^T_{\rho, k-1} = A^T_{k-2, \rho} A^T_{\rho, k-2}.$$
This proves (i) and (ii). To prove (iii) we begin with the following lemma.

\begin{lemma} \label{principality theorem}  One has  $A^{\sigma}_{k-2, \rho}=R^{\rm red}$ and $A^{\sigma}_{\rho,k-2}$ is a principal ideal of $R^{\rm red}$.  \end{lemma}
\begin{proof}  We have 
$$\sigma^{\rm red}(R^{\rm red}[G_{\Sigma}]) \cong \bmat R^{\rm red} & M_{1,2}(A^{\sigma}_{k-2, \rho}) & A^{\sigma}_{k-2, k-1}\\ M_{2,1}(A^{\sigma}_{\rho, k-2}) & M_{2,2}(R^{\rm red}) & M_{2,1}(A^{\sigma}_{\rho, k-1}) \\ A^{\sigma}_{k-1, k-2} & M_{1,2}(A^{\sigma}_{k-1, \rho}) & R^{\rm red} \emat.$$   Recall that $$\ov{\sigma}= \bmat \chi^{k-2} & a&b\\ & \rho & c \\ && \chi^{k-1}\emat,$$ with $a$ and $c$ both non-trivial classes. Thus a conjugate of $\sigma^{\rm red}$ which is adapted to the data of idempotents $\mE$ will have the corresponding properties. More precisely, $\sigma^{\rm red} \otimes \bfF$ will have a subquotient which is a non-split extension of $\rho$ by $\chi^{k-2}$ and one which is a non-split extension of $\chi^{k-1}$ by $\rho$. Thus we get that $A^{\sigma}_{k-2, \rho} = A^{\sigma}_{\rho, k-1}=R^{\rm red}$. 
Then by Chasles relations we get that also $A^{\sigma}_{k-2, k-1}=R^{\rm red}$. Finally, by Theorem \ref{1.4.4}(ii) we have that $A^{\sigma}_{ij} A^{\sigma}_{ji} \subset \fm$ for each pair $i \neq j$. Thus we get that each of $A^{\sigma}_{k-1, k-2}$, $A^{\sigma}_{k-1, \rho}$ and $A^{\sigma}_{\rho, k-2}$ is contained in $\fm$. 

By Lemma \ref{char} we have $A^{\sigma}_{k-1, k-2} = A^{\sigma}_{k-1, \rho}A^{\sigma}_{\rho, k-2} \subset \fm A^{\sigma}_{\rho, k-2}$.   This gives us  $R^{\rm red}$-module maps (note that $A^{\sigma}_{\rho, k-1}=R^{\rm red}$) $$\frac{A^{\sigma}_{\rho, k-2}}{A^{\sigma}_{\rho, k-1}A^{\sigma}_{k-1, k-2}} \twoheadrightarrow \frac{A^{\sigma}_{\rho, k-2}}{\fm A^{\sigma}_{\rho, k-2}} \implies \Hom_{R^{\rm red}}\left(\frac{A^{\sigma}_{\rho, k-2}}{\fm A^{\sigma}_{\rho, k-2}}, \bfF\right) \hookrightarrow \Hom_{R^{\rm red}}\left(\frac{A^{\sigma}_{\rho, k-2}}{A^{\sigma}_{\rho, k-1}A^{\sigma}_{k-1, k-2}}, \bfF\right).$$

We have (the first isomorphism following from 
our assumptions in part (iii) of Theorem \ref{4} and from Propositions \ref{torsion coeffs}(1) and \ref{Selmer groups 2} while the first injection following from Lemma \ref{injection1}) \begin{multline}\bfF\cong H^1_{\Sigma}(\bfQ, \Hom(\chi^{k-2}, \rho))\hookleftarrow\Hom_{R^{\rm red}}(A^{\sigma}_{\rho, k-2}/A'_{\rho, k-2}, \bfF) =\\ =\Hom_{R^{\rm red}}\left(\frac{A^{\sigma}_{\rho, k-2}}{A^{\sigma}_{\rho, k-1}A^{\sigma}_{k-1, k-2}}, \bfF\right) \hookleftarrow \Hom_{R^{\rm red}}\left(\frac{A^{\sigma}_{\rho, k-2}}{\fm A^{\sigma}_{\rho, k-2}}, \bfF\right),\end{multline} so, the (clearly non-trivial) $\bfF$-vector space $\frac{A^{\sigma}_{\rho, k-2}}{\fm A^{\sigma}_{\rho, k-2}}$ is one-dimensional. Thus, $A^{\sigma}_{\rho, k-2}$ is generated over $R$ by one element by Nakayama's lemma. 
\end{proof} 

\begin{lemma} \label{principality theorem 2}  One has  $A^{T}_{k-2, \rho}=R^{\rm red}$ and $A^{T}_{\rho,k-2}$ is a principal ideal of $R^{\rm red}$.  \end{lemma}

\begin{proof} As explained in the proof of Lemma \ref{5.2.1} this follows by applying $\phi_{k-2, \rho}$ (resp. $\phi_{\rho, k-2}$) to $A^{\sigma}_{k-2, \rho}$ (resp. $A^{\sigma}_{\rho,k-2}$) and using Lemma \ref{principality theorem} along with the fact that both $\phi$ maps are $R$-linear. \end{proof}

Lemma \ref{principality theorem 2} along with part (ii) of the Theorem imply part (iii). \end{proof}

\section{Sufficient conditions for $R^{\rm red}$ to be a discrete valuation ring} \label{Cyclicity}

Let $\rho: G_{\Sigma} \to \GL_2(\bfF)$ be a continuous absolutely irreducible and odd Galois representation of determinant $\chi^{2k-3}$, short crystalline at $p$ and such that for every prime $\ell\in\Sigma \backslash \{p\}$  one has $\rho|_{I_{\ell}} \cong \bmat 1&*\\ 0 & 1\emat\neq I_2$. 
Let $R_{\rho}$ be the universal deformation ring for deformations $\rho': G_{\Sigma} \to \GL_2(A)$ (here $A$ is an object of $\textup{LCN}(E)$) of $\rho$ which have determinant equal to $\epsilon^{2k-3}$,  are short crystalline at $p$ and minimal in the  sense that they satisfy  \be \label{min for rho} \rho'|_{I_{\ell}} \cong \bmat 1 &* \\ 0&1\emat.\ee 
We will assume that $R_{\rho} \cong \Oo$, i.e., is a discrete valuation ring (cf. \cite{BergerKlosin13}, section 6.1 for discussion of this assumption as well as section \ref{Examples} in the current paper). This implies that the set of strict equivalence classes of short crystalline, minimal deformations of $\rho$  to $\GL_2(\Oo)$ with determinant equal to $\epsilon^{2k-3}$ contains a single element. We will write  $\tilde{\rho}$ for a fixed representative of this strict equivalence class.

Consider short crystalline, semi-abelian $\ov{\sigma} : G_{\Sigma} \to \GL_4(\bfF)$ of the form $$\ov{\sigma}=\bmat \chi^{k-2} & a & b\\ & \rho & c \\ && \chi^{k-1}\emat$$ with $a,c$ giving rise to non-zero elements of $H^1_{f}(\bfQ, \rho(1-k))$ (cf. Corollary \ref{splitting at ell}).

 We assume in this section that $$ \dim_{\bfF}H^1_f(\bfQ, \rho(1-k)) =1.$$ 
For $A$ an object of ${\rm LCN}(E)$ we will call a matrix $M\in \GL_4(A)$ upper-triangular if  $$M=\bmat *&*&*&*\\ 0&*&*&*\\ 0&*&*&*\\ 0&0&0&*\emat.$$ We will say that a homomorphism $\sigma': G_{\Sigma} \to \GL_4(A)$ is upper-triangular if for all $g \in G_{\Sigma}$, the matrix $\sigma'(g) \in \GL_4(A)$ is upper-triangular. Finally, we say that a deformation $\sigma : G_{\Sigma} \to \GL_4(A)$ of $\ov{\sigma}$ is upper-triangular if there exists a member $\sigma'$ of the strict equivalence class of $\sigma$ which is upper-triangular. 
\begin{lemma} \label{5} Every short crystalline minimal  deformation $\sigma$ of $\ov{\sigma}$ to $\GL_4(A)$ such that $\tr \sigma = T_1 + T_2 +T_3$ with $T_1, T_2, T_3$ pseudocharacters such that $T_1 \equiv \chi^{k-2}$, $T_2 \equiv \tr \rho$, $T_3 \equiv \chi^{k-1}$ mod $\fm_A$  is upper-triangular. \end{lemma}
\begin{proof} This is Theorem 1.1 in \cite{Brown08}. Note that the Artinian assumption in \cite{Brown08} is unnecessary. See also remarks to Theorem 1 in \cite{Urban99}. \end{proof}
\begin{lemma} \label{diag pieces} Consider an upper-triangular deformation of $\ov{\sigma}$ to $\GL_4(\bfF[X]/X^2)$. Then every upper-triangular member $\sigma'$ of its strict equivalence class has the form $$\sigma' = \bmat \chi^{k-2} & * & *\\ & \rho & * \\ && \chi^{k-1}\emat.$$  \end{lemma} 
\begin{proof}  We need to show that the only deformations of $\chi^{k-2}$, $\rho$, $\chi^{k-1}$ are the trivial ones. Suppose that $\chi^{k-2} + \alpha X$ is a deformation of $\chi^{k-2}$ to $(\bfF[X]/X^2)^{\times}$. Then since $\chi^{2-k}$ is also   crystalline, we get that $1+\chi^{2-k}\alpha X$ is  crystalline. Then $\alpha':= \chi^{2-k} \alpha$ is a homomorphism from $G_{\Sigma}$ to the additive group $\bfF$. 
Let $\ell \mid N$ be a prime. Note that the total ring of fractions of $\bfF[X]/X^2$ is $\bfF[X]/X^2$. If $x \in G_{\Sigma}$ is such that $\sigma'(I_{\ell})$ is generated by $\sigma'(x)$ then by minimality we know that $\sigma'(x)-I_4$ is conjugate (over $\bfF[X]/X^2$) to $\bmat 0&0&0&0\\ 0&0&f&0\\0&0&0&0\\0&0&0&0\emat$ for $f \in \bfF^{\times}$.   Using the upper-triangular form of $\sigma'$ this implies that $\alpha$ must be unramified at $\ell$. Hence  the character $1 + \alpha' X: G_{\Sigma} \to (\bfF[X]/X^2)^{\times}$ can only be ramified at $p$. 
However, crystallinity forces the character to be trivial (by a simple modification of the proof of Lemma 9.6 in \cite{BergerKlosin13}),  so $\alpha=0$. Similarly one proves the only deformation of $\chi^{k-1}$ is the trivial one. Finally, the claim for $\rho$ follows from the assumption that $R_{\rho}$ is a dvr. \end{proof}
\begin{prop}\label{6} There do not exist any non-trivial upper-triangular deformations of $\ov{\sigma}$ to $\bfF[X]/X^2$. \end{prop}
\begin{proof} 
First note that $\bmat \chi^{k-2} & a\\ & \rho\emat$ is a subrepresentation of $\ov{\sigma}$  and that $\bmat \rho & c \\ & \chi^{k-1}\emat$ is a quotient of $\ov{\sigma}$ hence they are both short crystalline since $\ov{\sigma}$ is. Thus  (by Corollary \ref{splitting at ell}) 
 $a$ gives rise to a (non-zero) element of  $H^1_{f}(\bfQ, \Hom(\rho, \chi^{k-2}))=H^1_f(\bfQ, \rho(1-k))$  and $c$ gives rise to a (non-zero) element of  $H^1_{f}(\bfQ, \Hom(\chi^{k-1}, \rho))=H^1_f(\bfQ, \rho(1-k))$ (cf. Remark \ref{identification}). Since this last Selmer group is assumed to be one-dimensional, $a$ and $c$ give rise to linearly dependent cohomology classes.

By Lemma \ref{diag pieces} any upper-triangular deformation $\varrho$ of $\ov{\sigma}$ to $\GL_4(\bfF[X]/X^2)$ has the form $$\varrho=\bmat \chi^{k-2} & a+Xa' &b+ Xb'\\ & \rho & c+Xc' \\ && \chi^{k-1}\emat$$ for $a',b',c'$ valued in $\bfF$. 

Let us now show that all the off-block-diagonal entries are unramified outside $p$. Suppose $\ell \mid N$. By Corollary \ref{splitting at ell} we see that $a,b,c$ are all unramified at $\ell$. 
 Write $a'=\bmat a_1 & a_2\emat$, $c'=\bmat c_1 \\ c_2\emat$ with $a_i, c_i : G_{\Sigma} \to \bfF$. Let $x \in G_{\Sigma}$ be such that $\varrho(x)$ generates $\varrho(I_{\ell})$. After conjugating $\rho$ (if necessary), we may assume that $\rho(x) = \bmat 1&1\\ & 1 \emat$.  As the (2,3)-entry of $\varrho(x)-I_4$ lies in $\bfF^{\times}$ minimality of $\varrho$ forces the first row of $\varrho(x)-I_4$ to be a scalar multiple of the second row and the fourth column of $\varrho(x)-I_4$ to be a scalar multiple of the third column. This gives $a_1(x)=c_2(x)=0$. Set $E_0:= \bmat A \\ &D \emat$ with $A=\bmat 1&-a_2(x)X \\ & 1\emat$ and $D=\bmat 1 & c_1(x)X \\ & 1 \emat$. 
 One has $E_0 (\varrho(x)-I_4)E_0^{-1} = \bmat 0&0&0&b'(x)X\\ 0&0&*&0\\ 0&0&0&0\\0&0&0&0\emat.$ Note that $E_0 \equiv I_4$ (mod $\varpi$),  so the conjugate $E_0 \varrho E_0^{-1}$ is strictly equivalent to $\varrho$.  Hence we may assume that $a_2(x)=c_1(x)=0$. Since again the first row must be a scalar multiple of the second row we also get $b'(x)=0$.

Note that $\bmat \chi^{k-2} & a+Xa' \\ & \rho\emat$ is a subrepresentation of $\varrho$ (with $a+Xa'$ unramified away from $p$). Then using that $\dim_{\bfF}H^1_{f}(\bfQ, \rho(1-k))=1$ we obtain by Proposition 7.2 of \cite{BergerKlosin13} (where the assumption 6.1(ii) is satisfied for $\rho$ and for $\chi^{k-2}$ it can be replaced by the lack of a non-trivial deformation of $\chi^{k-2}$ to $(\bfF[X]/X^2)^{\times}$ - cf. Proof of Lemma \ref{diag pieces} above) that $\bmat \chi^{k-2} & a+Xa' \\ & \rho\emat$ can be conjugated (over $\bfF[X]/X^2$) to the representation $\bmat \chi^{k-2}&a \\ & \rho\emat$, i.e., there exists $A \in I_3 + XM_3(\bfF[X]/X^2)$  such that $A\bmat \chi^{k-2} & a+Xa' \\ & \rho\emat A^{-1} = \bmat \chi^{k-2} & a \\ & \rho\emat$. Then we have $$\bmat A&0\\ 0 & 1 \emat \varrho \bmat A &0\\0&1\emat^{-1} = \bmat \chi^{k-2} & a & b'' + Xb''' \\ & \rho & c''+Xc''' \\ && \chi^{k-1} \emat.$$ Since the image of $A$ in $M_3(\bfF)$ is the identity,  we see that $b''=b$ and $c''=c$, so in particular this is still a deformation of $\ov{\sigma}$. Hence we conclude that $\varrho$ is equivalent to $\bmat \chi^{k-2} & a& b + Xb' \\ & \rho & c+Xc' \\ && \chi^{k-1} \emat$
which we will still denote by $\varrho$.

Let $V$ be the representation space of $\varrho$ (i.e., the free $\bfF[X]/X^2$-module of rank 4 on which $G_{\Sigma}$ acts via $\varrho$). Let $W$ be the $G_{\Sigma}$-invariant free $\bfF[X]/X^2$-submodule of $V$ spanned by the vector ${}^t[1,0,0,0]$. Then $G_{\Sigma}$ acts on the quotient $V/W$ via $$\varrho':=\bmat \rho & c+Xc' \\ & \chi^{k-1}\emat.$$ Hence we can again conclude by Proposition 7.2 in \cite{BergerKlosin13} that there exists $A\in I_3+XM_3(\bfF[X]/X^2)$ such that  $A\varrho'A^{-1}=\bmat \rho & c \\ & \chi^{k-1}\emat$. Then we have $$\bmat 1 &0\\0&A\emat \varrho \bmat 1 &0\\0&A\emat^{-1} =\bmat \chi^{k-2} & a + Xa'' & b+ Xb''\\ & \rho & c \\ && \chi^{k-1} \emat. $$ Denote the right-hand side of the above equation by $\varrho'$. 
Now, let us re-write $\varrho'$  in the $\bfF$-basis of $(\bfF[X]/X^2)^4$ given by $$\mB=\left\{\bmat 1\\0\\0\\0\emat, \bmat 0\\1\\0\\0\emat, \bmat 0\\0\\1\\0\emat, \bmat 0\\0\\0\\1\emat, \bmat X\\0\\0\\0\emat, \bmat 0\\X\\0\\0\emat, \bmat 0\\0\\X\\0\emat, \bmat 0\\0\\0\\X\emat\right\}. $$ 
We get that $\varrho'$ in the basis $\mB$ has the form
$$\bmat \chi^{k-2} & a & b \\ & \rho & c \\ & & \chi^{k-1} \\ &a''&b'' & \chi^{k-2} & a & b\\ &&&&\rho & c \\ &&&&&\chi^{k-1} \emat.$$ Taking a sequence of submodules and quotients we get $$\bmat \chi^{k-2} & a & b \\ & \rho & c \\ & & \chi^{k-1} \\ &a''&b'' & \chi^{k-2} \emat \rightarrow \bmat  \rho & c \\   & \chi^{k-1} \\ a''&b'' & \chi^{k-2} \emat, $$ which if we conjugate it by $\bmat I_2&0&0\\ 0&0&1 \\ 0&1&0 \emat$ will give us \be \label{blue qt} \bmat \rho&&c \\ a'' & \chi^{k-2} & b'' \\ && \chi^{k-1}\emat,\ee from which we can extract a short crystalline 2-dimensional subrepresentation of the form $$\bmat \rho \\ a'' & \chi^{k-2}\emat,$$ and again (by crystallinity of the above representation) we see that $a''$
(which is unramified away from $p$ by an argument as before) 
gives rise to an element of $H^1_{f}(\bfQ, \Hom(\rho, \chi^{k-2})) = H^1_{f}(\bfQ, \rho(1-k))$. So, by one-dimensionality of the latter there exists a constant $\alpha \in \bfF$ such that $a'' = \alpha a$.

Hence we get that \be \label{varrho} \varrho =\bmat \chi^{k-2} & a & b+Xb' \\ &\rho & c + Xc' \\ && \chi^{k-1} \emat \cong \bmat   \chi^{k-2} & a+X\alpha a & b+Xb'' \\ &\rho & c  \\ && \chi^{k-1} \emat=\varrho',\ee 
 and the  isomorphism is by conjugation by an element of  $I_4 + X M_4(\bfF)$. Furthermore we note that $$M_0 \varrho' M_0^{-1} = \bmat \chi^{k-2} & a & b+Xb'''\\ & \rho & c \\ && \chi^{k-1}\emat,$$ where $$M_0= \bmat 1 \\ & (1+\alpha X)I_2 \\ && 1+\alpha X\emat \in 1+ XM_4(\bfF).$$ Hence we conclude that we can take $a''=0$. 

 Then
we see that $\bmat \chi^{k-2} & b'' \\  & \chi^{k-1}\emat$ is a quotient of the representation in \eqref{blue qt}. 
 Thus, $b''$ gives rise to an element of $H^1_{\Sigma}(\bfQ, \bfF(-1))$. 
By the argument from the beginning of the proof we see that $b''$ is unramified at all primes $\ell \mid N$, so it in fact gives rise to an element of $H^1_f(\bfQ, \bfF(-1))$ which is zero by Proposition \ref{class groups}.
 \end{proof}

 Let $I_{R'}\subset R'$ denote the total ideal of reducibility corresponding to the universal deformation $(\sigma')^{\rm univ}$ and $I_R \subset R$ the total ideal of reducibility corresponding to the universal deformation $\sigma^{\rm univ}$.
\begin{cor}\label{structure map} The structure maps $\Oo \to R'/I_{R'}$, $\Oo\to R/I_R$ and $\Oo \to R^{\rm red}/I^{\rm tot}$ are all surjective. \end{cor}
\begin{proof}  By Lemma 7.11 of \cite{BergerKlosin13} it is enough to show that the structure map $\Oo \to R'/I_{R'}$ is surjective. We need the following lemma.
\begin{lemma} \label{conj to upper} Let $I \subset R'$ be an ideal such that $R'/I \in {\rm LCN}(E)$. Then  $I \supset I'_R$ if and only if $(\sigma')^{\rm univ}$ mod $I$ is an upper-triangular deformation of $\ov{\sigma}$ to $\GL_4(R'/I)$.  \end{lemma} 
\begin{proof} The proof of Corollary 7.8 in \cite{BergerKlosin13} carries over to three Jordan-Hölder factors with the application of Theorem 7.7 in \cite{BergerKlosin13} replaced by an application of Lemma \ref{5} above. \end{proof}
The rest of the proof of Corollary \ref{structure map} is the same as the proof of Proposition 7.10 in \cite{BergerKlosin13} with the application of Corollary 7.8 in \cite{BergerKlosin13} now replaced with an application of Lemma \ref{conj to upper}. 
\end{proof}

\begin{thm} \label{Lvalue bound} Let $\tilde{\rho}$ be a representative of the unique deformation  of $\rho$ to characteristic zero. Suppose that $\#H^1_{f}(\bfQ, \tilde{\rho}(1-k)\otimes E/\Oo) \leq \#\Oo/L\Oo$ for some $L \in \Oo$. Then $\# R^{\rm red}/I^{\rm tot}\leq \#\Oo/L\Oo$. \end{thm}

\begin{proof} 
Again by Lemma 7.11 of \cite{BergerKlosin13} it is enough to show that $\#R'/I_{R'} \leq \#\Oo/L\Oo$.
First let us note that it follows from the proof of Lemma \ref{diag pieces} that there do not exist any non-trivial short crystalline deformations of $\chi^{k-2}$ and of $\chi^{k-1}$ to $(\bfF[X]/X^2)^{\times}$. Let $R_{k-2}$ (resp. $R_{k-1}$) denote the universal short crystalline deformation ring of $\chi^{k-2}$ (resp. $\chi^{k-1}$). Then we can adapt the proof of Proposition 7.10 in \cite{BergerKlosin13} with $S$ there replaced by $R_{k-2}$ to show that $R_{k-2}/\varpi R_{k-2}=\bfF$ and hence by Nakayama's lemma the structure map $\Oo \to R_{k-2}$ is onto. However, since $\epsilon^{k-2}$ is a deformation of $\chi^{k-2}$ to $\Oo^{\times}$, we must in fact have $R_{k-2}=\Oo$ and so $\epsilon^{k-2}$ is a unique such deformation. The same conclusion holds for $R_{k-1}$. 
Assume that $\#R'/I_{R'} =\#\Oo/\varpi^s \Oo> \#\Oo/L\Oo$ (we allow $s=\infty$ here). Then by Lemma \ref{conj to upper} there exists an upper-triangular  deformation $\sigma: G_{\Sigma} \to \GL_4(\Oo/\varpi^s\Oo)$ of $\ov{\sigma}$. By the above argument and using the fact that $\tilde{\rho}$ is the unique deformation of $\rho$ to $\GL_2(\Oo)$ the deformation $\sigma$ must have the form $$\bmat \epsilon^{k-2} & *_1 & *_2 \\ & \tilde{\rho} & *_3 \\ && \epsilon^{k-1} \emat,$$ where the diagonal pieces are understood to be taken mod $\varpi^s$. 

Arguing as in the proof of Proposition \ref{6} 
we see that minimality guarantees that the classes of $*_1$ and $*_3$ are unramified away from $p$ and thus give rise to classes in $H^1_f(\bfQ, \tilde{\rho}(1-k)\otimes E/\Oo[\varpi^s]) = H^1_f(\bfQ, \tilde{\rho}(1-k)\otimes E/\Oo)[\varpi^s]$ (cf. Proposition \ref{torsion coeffs}(1)).   Since the reduction $\ov{\sigma}$ of $\sigma$ has the property that its entry corresponding to $*_3$ above gives rise to a non-trivial element in $H^1_{f}(\bfQ, \rho(1-k))$, we conclude that the image of $*_3$ is not contained in $\varpi \Oo/\varpi^s \Oo$, hence gives rise to an element of $H^1_{f}(\bfQ, \tilde{\rho}(1-k)\otimes E/\Oo)$ 
 of order $\#\Oo/\varpi^s$, which is a contradiction.
\end{proof}

\begin{cor} \label{traces} The rings $R'$, $R$ and $R^{\rm red}$ are all topologically generated as an $\Oo$-algebra by the set $\{\tr \sigma (\Frob_l) \mid l \nmid Np\}$, where $\sigma$ stands for $(\sigma')^{\rm univ}$, $\sigma^{\rm univ}$ and $\sigma^{\rm red}$ respectively. \end{cor}
\begin{proof} It is enough to prove this for $R'$ as the other rings are quotients of it. For this one can use the same proof as the one of Proposition 7.13 in \cite{BergerKlosin13} replacing again the application of Corollary 7.8 with Lemma \ref{conj to upper}. \end{proof}

The upshot of this section is the following theorem identifying conditions when the universal deformation ring is a  discrete valuation ring. To make the statement self-contained we will include all the assumptions made so far.
\begin{thm} \label{dvrness} Let
$\rho_f : G_{\Sigma} \to \GL_2(E)$ be the Galois representation attached to a newform $f \in S_{2k-2}(N)$ for $N$ a square-free integer with $p \nmid N$ and $\Sigma=\{ \ell \mid N \} \cup \{p\}$ such that $p \nmid 1+w_{f, \ell}\ell$ for all primes $\ell \mid N$. Suppose also that the residual representation $\rho:=\ov{\rho}_f$ is absolutely irreducible and ramified at every prime $\ell \mid N$. Assume that $\dim_{\bfF}H^1_f(\bfQ, \rho(2-k))=1$ and that $\# H^1_f(\bfQ, \rho_f (1-k)\otimes E/\Oo) \leq \# \bfF$. Furthermore assume that $R_{\rho}$ is a dvr. 
Let $\ov{\sigma}: G_{\Sigma} \to \GL_4(\bfF)$ be a continuous representation  short crystalline at $p$ and semi-abelian of the form $$\ov{\sigma}=\bmat \chi^{k-2} & a & b\\ & \rho & c \\ && \chi^{k-1}\emat$$ with $a,c$ giving rise to non-zero elements of $H^1_{f}(\bfQ, \rho(1-k))$. Suppose $\ov{\sigma}$ admits a deformation $\sigma$ to $\GL_4(\Oo)$ which is $\tau$-self-dual. 

  Then $\sigma$ gives rise to an $\Oo$-algebra isomorphism $R^{\rm red}_{\ov{\sigma}}=R^{\rm red} \xrightarrow{\sigma}\Oo$. \end{thm}
\begin{proof}  The $\Oo$-algebra map $R \twoheadrightarrow \Oo$ induced by $\sigma$ factors through $R^{\rm red}$ since $\Oo$ is torsion-free. Theorem \ref{4} combined with Theorem \ref{Lvalue bound} imply that the maximal ideal of $R^{\rm red}$ is principal. Because of the surjection to $\Oo$ its generator is not nilpotent. As in \cite{Calegari06} Lemma 3.4 we can therefore deduce that $R^{\rm red}$ is a discrete valuation ring and the surjection an isomorphism $R^{\rm red}\cong \Oo$.  \end{proof}

\section{Application to the Paramodular Conjecture} \label{Application to the Paramodular Conjecture}

 The following conjecture is due to Brumer and Kramer (cf. Conjecture 1.1.1 in  \cite{Brumeretal19preprint}) and is often referred to as the Paramodular Conjecture.

\begin{conj}[Brumer - Kramer]  \label{conj:paramodular}  For every isogeny class of abelian surfaces $A_{/\bfQ}$ of conductor $N$ with $\End_{\bfQ} A = \bfZ$ there exists a  weight 2 Siegel modular form $F$,  which is not in the space spanned by the Saito-Kurokawa lifts, has level $K(N)$ and rational eigenvalues,  where $K(N)$ is the paramodular group of level $N$ defined by $K(N) = \gamma M_{4}(\bfZ) \gamma^{-1} \cap \Sp_4(\bfQ)$ with $\gamma = \diag[1,1,1,N]$.  The $L$-series of $A$ and $F$ should agree and the $p$-adic representation of $T_{p}(A) \otimes \bfQ_{p}$ should be isomorphic to the one associated to $F$ for any $p$ prime to $N$ where $T_{p}(A)$ is the $p$-adic Tate module.
\end{conj}

In this section we will show how our results can be used to verify this conjecture in some cases when $A$ has rational $p$-torsion. More precisely, our method will only allow for the verification of the second claim, i.e., the existence of the isomorphism of the Tate module with the representation associated to $F$. 

Let $p>2$ be a prime and let $A$ be an abelian surface of square-free conductor $N$ as in Conjecture \ref{conj:paramodular}. Suppose that $p \nmid N$ and that $A$ has a 
polarization of degree prime to $p$ and a $\bfQ$-rational point of order $p$. Then the $p$-adic Tate module gives rise to a Galois representation   $$\sigma_A:=V_p(A): G_{\Sigma} \to \GL_4(\bfQ_p),$$ where $\Sigma=\{p\} \cup \{\ell \mid N\}$, short crystalline at $p$  (by \cite{SerreTate68} and \cite{Fontaine82}), which is absolutely irreducible since it is semisimple and  $\End_{\bfQ_p[G_{\bfQ}]}(V_{p}(A))=\End_{\bfQ}(A) \otimes \bfQ_p=\bfQ_p$ by \cite{Faltings83}.   
Furthermore, under our assumptions the semisimple residual representation $\ov{\sigma}_A^{\rm ss}: G_{\Sigma} \to \GL_4(\bfF_p)$ has the form $$\ov{\sigma}_A^{\rm ss} = 1 \oplus \rho^{\rm ss} \oplus \chi,$$ where, as before, $\chi$ is the mod $p$ cyclotomic character and $\rho$ is a two-dimensional representation. 
The representation $\sigma_A$ has determinant $\epsilon^2$ and hence $\det \rho=\chi$, so in particular, $\rho$ is odd. 

We will assume that $\rho$ is absolutely irreducible, ramified at all $\ell \mid N$. It follows from Serre's conjecture  
 \cite{KhareWintenberger09} (see Proposition 4 of \cite{Serre87} for the determination of the Serre weight) that $\rho \cong \ov{\rho}_f$ where $\rho_f: G_{\Sigma} \to \GL_2(\bfQ_p)$ is the Galois representation attached to a newform $f \in S_2(N)$ and $\ov{\rho}_f$ denotes its mod $p$ reduction.

We now assume that $A$ has semi-abelian reduction at all $\ell \mid N$, i.e. the reduction is an extension of an elliptic curve by a torus  and that $p$ does not divide the Tamagawa number at $\ell$, i.e. the number of connected components of the special fiber of a Neron model of $A$ at $\ell$.   Under these assumptions \cite{Duff98} Proposition 1.3 proves that $\ov{\sigma}_{A}|_{I_{\ell}} \cong \exp(t_{\ell}N_1)$.  The proof generalizes to $A[p^n]$ as \cite{SerreTate68} Lemma 1 and 2 hold in this generality, so we deduce that $\sigma_{A}|_{I_{\ell}} \cong \exp(t_{\ell}N_1)$.

For a positive integer $N$ and an integer $k\geq 2$ we will denote by $\mS_k(N)^{\rm para}$ the space of Siegel modular forms of genus 2, (parallel) weight $k$ and paramodular level $N$. 
For a prime $\ell \nmid N$, write $T_{\ell, 2}$ for the Hecke operator given by the double coset $K(N) \diag(1,1, \ell, \ell) K(N)$, $T_{\ell,1}$ for the Hecke operator given by the double coset $K(N) \diag(1, \ell, \ell^2, \ell) K(N)$, and $T_{\ell,0}$ for the Hecke operator given by the double coset $\ell K(N)$ (see e.g. p. 18 of \cite{Pilloni17preprint}). We will say that $F \in \mS_k(N)^{\rm para}$ is an eigenform if $F$ is an eigenform for $T_{\ell, i}$ for all $\ell \nmid N$ and $i=0,1,2$. If $F$ is an eigenform we will write $\lambda_{\ell,i}(F)$ for the eigenvalue of $T_{\ell,i}$ corresponding to $F$.

\begin{thm}[Taylor, Weissauer, Laumon]  \label{Gal rep for Sieg} 
For $k \geq 2$ let $F \in S_k(N)^{\rm para}$ be an eigenform. There exists a finite extension $E$ of $\bfQ_p$ and a continuous semisimple representation $$\sigma_F: G_{\Sigma} \to GL_4(E)$$ with $\sigma_F^{\vee} \cong \sigma_F \otimes \epsilon^{3-2k}$ such that for all primes $\ell \nmid Np$ one has $\det (I_4 -X\sigma_F(\Frob_{\ell}))^{-1} = Q_{\ell}(X)$ for the Hecke polynomial $$Q_{\ell}(X)=1- \lambda_{\ell,2}(F) X + \ell(\lambda_{\ell,1}(F) + (1+\ell^2) \lambda_{\ell,0}(F))X^2 - \ell^3\lambda_{\ell,2}(F) \lambda_{\ell,0}(F)X^3 + \ell^6 \lambda_{\ell,0}(F)^2 X^4,$$  where $\lambda_{\ell,0}(F)=\ell^{2k-6}$. 
\end{thm} 

\begin{rem}
This result is well-known, but for the benefit of the reader we explain the relationship to results in the literature. The eigenform $F$ generates an automorphic representation $\pi$ that decomposes into a direct sum $\pi_1 \oplus \ldots \oplus \pi_n$ of irreducible cuspidal automorphic representations of $\GSp_4$ whose local representations agree away from $N$. Laumon \cite{Laumon05} and Weissauer \cite{Weissauer05} constructed the Galois representations for irreducible automorphic representations $\pi$ with $\pi_{\infty}$ holomorphic discrete series (and hence for eigenforms $F$ as above for weight $k>2$ as all $\pi_i$ give rise to the same semi-simple Galois representation). The result for $k=2$ can be deduced from this using Taylor's argument in Example 1 of section 1.3 in \cite{Taylor91}. \cite{Jorza10} and \cite{Mok14} Proposition 4.14 provide an alternative proof (for \cite{Mok14} under the restriction that 
$F$ be of type (G) in the sense of  \cite{Schmidt16} section 2.1, i.e. such that all $\pi_i$ have a cuspidal transfer to $\GL_4$, but see Remark \cite{Mok14} 4.15). 
\end{rem}

\begin{thm}[\cite{ChaiFaltings90}, \cite{Jorza12} Theorem 3.1] \label{Jorza} Let $F$ be as in Theorem \ref{Gal rep for Sieg}, $k \geq 2$ and $p \nmid N$.  If $k=2$ assume that  $F$ has pairwise distinct roots for the Hecke polynomial    $Q_p(F)$.  Then $\sigma_F$ is crystalline at $p$  and short crystalline if $p>2k-2$. \end{thm}

\begin{prop} \label{prop8.4}
For $k\geq 2$ and $N$ squarefree let $F\in S_k(N)^{\rm para}$ be an eigenform of type $(G)$ in the sense of  \cite{Schmidt16} section 2.1. 
If $k=2$ we further assume that $\sigma_F$ is absolutely irreducible and that $\ov{\sigma}_F^{\rm ss}$ is ramified at $\ell$ 
for all $\ell \mid N$.
Then we have $$\sigma_F|_{I_{\ell}} \cong \exp(t_p N_1) \text{ for all } \ell \mid N.$$
\end{prop}

\begin{proof}
For $k>2$ this is proved by \cite{Mok14} Theorem 3.5 (local-global compatibility up to Frobenius semisimplification) and \cite{Sorensen10} Corollary 1 (monodromy rank 1). (We know by \cite{Schmidt16} Proposition 1.2.1 that type $(G)$ paramodular eigenforms give rise to generic local representations, so that Sorenson's result applies.)

As for $k=2$ \cite{Mok14} Theorem 4.14 only proves local-global compatibility up to semisimplification we argue as follows: Theorem 1.1 of \cite{MokTan15} allows us to $p$-adically approximate the Hecke eigenvalues of a weight $2$ form $F$ by those of cohomological Hecke eigenforms $F_n$. This means that the corresponding Galois representations $\sigma_{F_n}$ are trace convergent to $\sigma_F$ in the sense of  \cite{BCKL05}, i.e. $\tr(\sigma_{F_n})(g)$ converges to  $\tr(\sigma_{F})(g)$ in ${\bfC}_p$ for all $g \in G_{\bfQ}$. As we assume that $\sigma_F$ is absolutely irreducible we can invoke Theorem 1.2 of \cite{BCKL05} to deduce that  $\sigma_{F_N}$ are physically convergent to $\sigma_F$, i.e. such that there exist $\bfC_p$-bases  for $\sigma_{F_n}$ and $\rho$ such that the matrix entries of $\sigma_{F_n}$ converge to the corresponding entries of $\sigma_{F}$. This result also tells us that the $\sigma_{F_n}$ are absolutely irreducible for $n \gg 0$ so we know that the $F_n$ are eventually of type $(G)$. As these are of weight $k>2$
we know that $\sigma_{F_n}|_{I_{\ell}} \cong \exp(t_p N_1)$ and the rank of the monodromy remains 1 in the limit by our assumption that $\ov{\sigma}_F^{\rm ss}$ is ramified.
\end{proof}

\begin{thm} \label{paramodularity} Let $A$ and $\rho=\ov{\rho}_f$ be as above and assume that $p \nmid 1+w_{f, \ell} \ell$ for all $\ell \mid N$, $R_{\rho}=\Oo$, $\#H^1_f(\bfQ,\ov{\rho}_f )= \#\bfF$ and $\#H^1_f(\bfQ, \rho_f(-1) \otimes E/\Oo)\leq \#\bfF$. 

Furthermore suppose that there exists
 $F \in \mS_2(N)^{\rm para}$ such that $\lambda_{\ell, 2}(F) \equiv 1 + \ell + a_{\ell}(f)$ (mod $p$) for all primes $\ell \nmid Np$ and such that its Hecke polynomial at some prime $\ell \nmid Np$  vanishes neither at $1$ nor at $1/\ell$. We also assume that  $F$ has pairwise distinct roots for the Hecke polynomial  $Q_p(F)$. 
 Then $\sigma_A \cong \sigma_F$. In particular, $A$ is  paramodular of level $N$. \end{thm}

\begin{proof} The congruence,  Tchebotarev's density Theorem and the Brauer-Nesbitt Theorem imply that the semisimplification $\ov{\sigma}_F^{\rm ss}$ of the  mod $p$ reduction  of $\sigma_F$ has the form $\ov{\sigma}_F^{\rm ss}=1 \oplus \ov{\rho}_f \oplus \chi=\ov{\sigma}^{\rm ss}_A$. 

We will show below that $\sigma_F$ is absolutely irreducible. Using Corollary \ref{corlat} (and the fact that $\sigma_F$ is short crystalline at $p$ - cf. Theorem \ref{Jorza}) we can then choose  $G_{\Sigma}$-stable lattices in  $\sigma_A$ and in the representation space of $\sigma_F$  so that with respect to these lattices $\ov{\sigma}_A$ and $\ov{\sigma}_F$ have the form as in Corollary \ref{corlat}. Then by Proposition \ref{uni1} we get that in fact $\ov{\sigma}_A \cong \ov{\sigma}_F$. By adjusting the basis of $\sigma_F$ if necessary we can assume that $\ov{\sigma}_A=\ov{\sigma}_F$, hence  we obtain that $\sigma_F$ is a deformation of $\ov{\sigma}_A$. Let $R$ be the quotient of the universal deformation ring of $\ov{\sigma}_A$ as in Proposition \ref{reprsd}. Then by Theorem \ref{dvrness} we see that $\sigma_A$ and $\sigma_F$ both give rise to an isomorphism $R^{\rm red} \cong \Oo$ of $\Oo$-algebras.  Hence we get a commutative diagram $$\xymatrix{\Oo \ar[ddr]_{\textup{id}}\ar[dr]\ar[drr]^{\textup{id}}\\ & R^{\rm red} \ar[d]^{\sigma_F} \ar[r]_{\sigma_A}^{\sim} & \Oo \\ & \Oo}$$ from which we see that $\sigma_A=\sigma_F$ as maps from $R^{\rm red}$ to $\Oo$, which implies that the representations $\sigma_F$ and $\sigma_A$ are isomorphic since they are both composites of the universal deformation with the map $R^{\rm red} \to \Oo$. 

Let us now show that $\sigma_F$ is absolutely irreducible.  Indeed, note that $\sigma_F$ cannot be the sum of 4 characters because its reduction has an absolutely irreducible two-dimensional component $\ov{\rho}_f$. Suppose that $\sigma_F^{\rm ss}$ splits as a direct sum of two irreducible 2-dimensional representations $\sigma_1 \oplus \sigma_2$. Note that $\sigma_1^{\vee} \not \cong \sigma_2 \otimes \epsilon^{-1}$ since without loss of generality $\ov{\sigma}_1=\rho \neq \ov{\sigma}_2^{\rm ss}=1 \oplus \chi$. We therefore have $\sigma_1^{\vee} \cong \sigma_1 \otimes \epsilon^{-1}$, which implies $\det(\sigma_1)=\det(\sigma_2)=\epsilon$. Indeed we must have $\det(\sigma_1)=\epsilon \phi$ for some quadratic character $\phi$ reducing to the identity mod $p$ since $\det(\ov{\sigma}_1)=\chi$. Since $p>2$  the character $\phi$ has to be trivial. 

We can therefore argue as in case (v) on p. 46 of \cite{SkinnerUrban06}. Note that $\sigma_1$ has to be odd since $\rho$ is, so $\sigma_2$ must be odd as well. Since the automorphic representation $\pi$ corresponding to $F$ can be transferred to an isobaric automorphic representation on $\GL_4$ the assumptions in Theorem C of \cite{Ramakrishnan13} are satisfied, which tells us that $L^S(\sigma_1^{\vee} \otimes \sigma_2,1) \neq 0$ for some finite set of places $S$. This means that the standard $L$-function $L^S(s,\pi, {\rm std})=\zeta^S(s) L^S(\sigma_1^{\vee} \otimes \sigma_2,s)$ has a pole at $s=1$, which by \cite{KudlaRallisSoudry92}  implies that $\pi$ is endoscopic, i.e. corresponds to type $(Y)$ in \cite{Schmidt16}. Since $F$ is paramodular this contradicts \cite{Schmidt16} Lemma 2.2.1. 
This excludes the case $\sigma_A=\sigma_1 \oplus \sigma_2$.

 It remains to show that $\sigma_A$ cannot split as $\sigma_1 \oplus \chi_1$ with $\sigma_1$ a 3-dimensional representation and $\chi_1$ a character. If it did, then the spin $L$-function of $A$ would have a linear factor. Since the only short crystalline, minimal deformations of $1$ (resp. $\chi$) are $1$ (resp. $\epsilon$) - cf. the beginning of the proof of Lemma \ref{diag pieces} - we must have that $\chi_1$ is either $1$ or $\epsilon$. Thus the linear factor of the local spin $L$-function must be $1-\ell^{-s}$ (if $\chi_1=1$) or $1-\epsilon(\ell)\ell^{-s}=1-\ell \ell^{-s}$. But the local spin $L$-factor is the local Hecke polynomial at $\ell$ with $X$ replaced by $\ell^{-s}$, so we're done by our assumption.
  \end{proof}

\section{Examples} \label{Examples}

In this section we work out in detail how our result  proves paramodularity of an abelian surface of conductor 731. We also discuss  other cases where the result may be applicable without going into details.

\subsection{Conductor $N=731$}

 We will show that  for an abelian surface of this conductor  (using a congruence result proved in the Appendix) the conditions for Theorem \ref{paramodularity} are satisfied. This establishes a new case of the Paramodular Conjecture.

For $N=731=17\cdot 43$
there exists an abelian surface $A$ of conductor $N$, which arises as the Jacobian of the hyperelliptic curve with equation $$y^2+(x^3+x^2)y=x^5+2x^4-x-3.$$ According to LMFDB \cite{lmfdb:731.a.12427.1} it has semi-abelian reduction at 17 and 43 and has a rational point of order $p=5$. 
We first note that $5 \nmid 1 \pm \ell$ for $\ell=17,43$.

Since $A$ has rational 5-torsion  and is principally polarised  we know that $\ov{\sigma}_{A,p}^{\rm ss}=1 \oplus \rho \oplus \chi$. Furthermore $A$ has semi-abelian reduction and has Tamagawa numbers at 17 and 43 not divisible by $5$, 
which tells us that if $\rho$ were reducible  over any finite extension of $\bfQ_p$ its semisimplification would be unramified away from 5, i.e. of the form $\chi^i \oplus \chi^{1-i} \mod{5}$ for $i=1$ or $2$. We check  on specific Frobenius elements that the trace of $\rho$ is not equal to $\chi^i + \chi^{1-i}$ for $i=1,2$ so that $\rho$ has to be absolutely irreducible. Serre's conjecture implies that $\rho$ is modular by a  cuspidal weight 2 eigenform of level 17, 43 or 731 (see section \ref{Application to the Paramodular Conjecture}).  Using MAGMA \cite{MAGMA97} we check (by comparing the linear term of the Euler factor  of $A$ at $\ell$ with $1+\ell+a_{\ell}(f)$) that the only newform giving rise to $\rho$ is the modular form corresponding to the elliptic curve $E$ (Cremona label 731a1) of conductor 731 given by $$y^2+xy+y=x^3-539x+4765,$$ which has rank 1 and trivial rational torsion. This also shows that 5 is not a congruence prime for the corresponding weight 2 modular form, i.e. that $R_{\rho}=\bfZ_p$.
By consulting LMFDB \cite{lmfdb:731.a1} we know that $\ov{\rho}_{E,p}$ is  ramified at 17 and 43 (since $E$ has non-split reduction at both primes). We also note that $E$ has good ordinary reduction at $5$.

We can check the Selmer group conditions as follows: We have 
$H^1_f(\bfQ, \ov{\rho}_{E,p}) = H^1_f(\bfQ, {\rho}_{E,p} \otimes \bfQ_p/\bfZ_p)[p]={\rm Sel}_p(E)[p]$ where the first equality follows from 
Proposition \ref{torsion coeffs} and for the second one see e.g. \cite{Rubin00} Proposition 1.6.8. Here ${\rm Sel}_p(E)$ is defined by 
\be \label{Selmer sequence} 0 \to E(\bfQ)\otimes \bfQ_p/\bfZ_p \to {\rm Sel}_p(E) \to \Sh(E)[p^{\infty}] \to 0.\ee Using snake lemma on the diagram:
$$\xymatrix{0 \ar[r] &E(\bfQ)\otimes \bfQ_p/\bfZ_p \ar[r]\ar[d]^{\cdot p}& {\rm Sel}_p(E) \ar[r]\ar[d]^{\cdot p}& \Sh(E)[p^{\infty}] \ar[r]\ar[d]^{\cdot p} &0 \\
0 \ar[r] &E(\bfQ)\otimes \bfQ_p/\bfZ_p \ar[r] & {\rm Sel}_p(E) \ar[r]& \Sh(E)[p^{\infty}] \ar[r] &0}$$ we get an exact sequence  $$0 \to  E(\bfQ)\otimes \bfQ_p/\bfZ_p[p] \to {\rm Sel}_p(E) [p] \to \Sh(E)[p] \to 0.$$ Since $\ov{\rho}_{E, p}$ is absolutely irreducible, we conclude that $E$ has no rational $p$-torsion and hence $(E(\bfQ)\otimes \bfQ_p/\bfZ_p)[p] = E(\bfQ)/pE(\bfQ)$. This implies that \eqref{Selmer sequence} gives rise to a short exact sequence $$0 \to E(\bfQ)/p E(\bfQ) \to  H^1_f(\bfQ, \ov{\rho}_{E,p})\to \Sh(E)[p] \to 0,$$

Since the rank of $E$ is 1 we therefore deduce that $\#H^1_f(\bfQ, \ov{\rho}_{E,p})=5$ if $ \Sh(E)[5]=0$. According to LMFDB the analytic order of $\Sh(E)$ is 1. By Skinner et al. \cite{JSW17} the $p$-part of the BSD formula is satisfied, so $ \Sh(E)[5]=0$  (alternatively this was checked by  Grigorov, Jorza, Patrikis, Stein, and Tarnita, see Theorem 3.27 in \cite{GJPST09}). 

It now remains to bound the order of  $\#H^1_f(\bfQ, {\rho}_{E,p}(-1) \otimes \bfQ_p/\bfZ_p)$ (which we know to be non-trivial) from above by $5$. 
Setting $f \in S_2(\Gamma_0(731))$  in Proposition \ref{prop3.8} to be the newform corresponding to $E$ it is enough to calculate $p$-valuation of the $p$-adic $L$-series $L(E,\omega^{-1})\in \bfZ_p[[T]]$ specialized at $T=p$. 
Using SAGE we confirm that this $p$-valuation is indeed 1.

As proved in the Appendix there is a unique non-lift paramodular  eigenform $F$ new at level 731 (non-vanishing modulo $5$  and with Hecke eigenvalue $\lambda_{5,2}(F)=0$) that is congruent modulo $5$ to a Gritsenko lift (paramodular Saito-Kurokawa lift) of a modular form of level dividing 731. 
This congruence of Fourier expansions implies a congruence of Hecke eigenvalues modulo $5$. 

Using MAGMA we can rule out that $F$ is congruent to all but the Saito-Kurokawa lift of the rational modular form $f$ corresponding to $E$. 
Since $\lambda_{5,2}(F)=0$ its Hecke polynomial at $5$ is of the form $1+cT^2+25T^4$ for some $c \in \bfZ$, which either has distinct roots or more than one repeated root. As $a_5(f) \equiv -1 \mod{5}$ the Hecke polynomial of the Saito-Kurokawa lift of $f$ (given by $(1-\alpha T)(1- \beta T) (1-T) (1-5T)$ with $\alpha+\beta=a_5(f)$) is congruent to $1-T^2$ modulo $5$, so we deduce that the roots of the Hecke polynomial of $F$ are distinct modulo 5. 

Theorem \ref{appendixtheorem}(4) in the appendix also shows that the Hecke polynomial at 2 is given by $4 T^4 + 2 T^3 + 2 T^2 + T + 1= (1 - T + 2 T^2) (1 + 2 T + 2 T^2)$.  Its complex roots  do not include $1, 2$ or $1/2$.

\subsection{Other examples} \label{s9.2}
Poor and Yuen have found candidate paramodular forms for some other examples of abelian surfaces. For the seven conductors in Table 5 of \cite{PoorYuen15} one can check that only the abelian surface of conductor $N=277$ involves a congruence (for the torsion prime $p=3$, but not for $p=5$) with the paramodular Saito-Kurokawa lift of a modular form corresponding to an elliptic curve. 
In this case the $p$-valuation of the $p$-adic $L$-series gives us a bound of $9$ on $\#H^1_f(\bfQ, \rho_f(-1) \otimes E/\Oo)$. It may still be possible that the Bloch-Kato Selmer group has order 3, but we were not able to confirm this (and this would fall outside of the range of applicability of our results as we exclude $p=3$).  However, this abelian surface has been proved to be paramodular by Brumer et al. using the Faltings-Serre method, but a proof has not yet  appeared in print. 

When the conjectured congruence is with the Saito-Kurokawa lift of a modular form with non-rational Fourier coefficients it is sometimes possible to check that $R_{\ov{\rho}_f}$ is a discrete valuation ring.  This is the case for the examples of conductor $N=349, 353$ and $389$ on the list in \cite{PoorYuen15}. These surfaces have rational $p$-torsion for $p=13, 11$ and $5$, respectively.  We have, however,  not tried to check the Selmer group assumptions for these cases. 

Another example we want to highlight is that of conductor $N=997$, where there exists an abelian surface with $3$-torsion. In this case the expected congruence is between a paramodular non-lift and a Saito-Kurokawa lift of a modular form corresponding to a rank 2 elliptic curve. As the root number of the modular form is $+1$ this Saito-Kurokawa lift has to be of congruence level, rather than paramodular as the other cases.
It would be interesting to confirm the existence of a  matching  paramodular non-lift Siegel modular form in this case, but Theorem \ref{paramodularity} does not apply because  $p=3$ and the elliptic curve rank >1 should imply $\#H^1_f(\bfQ, \ov{\rho}_{E,p})>3$.

\section{Modularity theorem for cohomological weights $k$} \label{Modularity theorem}
In this section we will prove a modularity theorem in cases $k>2$  and $N=1$. The sole reason for these two restrictions is that at present only in this case we have a result (which is due to Brown) providing us with enough congruences among Siegel modular forms \cite{Brown11}. This result also allows us to replace the assumption that $\#H^1_f(\bfQ, \rho_f(1-k)\otimes E/\Oo) \leq \#\bfF$ with the weaker assumption that $\dim_{\bfF} H^1_f(\bfQ, \rho_f(1-k))=1$. In particular, the proof does not proceed via proving that $R^{\rm red}$ is a discrete valuation ring - in fact the latter property is not implied by our $R=T$ theorem. We again include all the assumptions to make the statement self-contained, however to make their use more transparent we will separate the assumptions that are necessary for Brown's congruence result (collecting them below as Assumption \ref{Jim1} - cf. Theorem 5.4 and Corollary 5.6 in \cite{Brown11}) from the assumptions required on the deformation side (which will be spelled out in the statement of Theorem \ref{modularity}). 

Let $k$ be a positive integer and $p$ a prime such that $p > 2k-2$. 
\begin{assumption} \label{Jim1} Suppose $k>9$ is even. Assume also that there exists $N \in \bfZ_{>1}$ and a fundamental discriminant $D<0$ such that $\chi_D(-1)=-1$ and $p \nmid N D [\Sp_4(\bfZ):\Gamma_0^{(2)}(N)]$.  Here $\Gamma_0^{(2)}(N)$ is the congruence subgroup of $\Sp_4(\bfZ)$ of level $N$. Let $f \in S_{2k-2}(1)$ be a $p$-ordinary newform. Also assume that there exists a Dirichlet character $\psi$ of conductor $N$ such that $$\val_{\varpi}(L^{\Sigma}(3-k, \psi) L^{\rm alg}(k-1, f, \chi_D) L^{\rm alg}(1,f,\psi) L^{\rm alg}(2,f,\psi))=0.$$ For precise definitions of the $L$-factors, cf. \cite{Brown11}, section 5.
\end{assumption}

\begin{thm} \label{modularity} Assume Assumption \ref{Jim1}.
Let
$\rho_f : G_{\Sigma} \to \GL_2(E)$ be the Galois representation attached to the newform $f$ and write $\rho$ for the residual representation $\ov{\rho}_f$, which we assume to be absolutely irreducible. Here $\Sigma= \{p\}$. Assume that $$\dim_{\bfF}H^1_f(\bfQ, \rho(1-k))=\dim_{\bfF}H^1_f(\bfQ, \rho(2-k))=1.$$   Furthermore assume that $R_{\rho}$ is a dvr. 
Let $\sigma: G_{\Sigma} \to \GL_4(E)$ be an absolutely irreducible short crystalline Galois representation with $\det \sigma = \epsilon^{4k-6}$, satisfying $\sigma^{\vee} \cong \sigma(3-2k)$  and such that its residual representation $\ov{\sigma}$ (defined only up to semisimplification) satisfies $$\ov{\sigma}^{\rm ss} \cong \chi^{k-2} \oplus \rho \oplus \chi^{k-1}.$$  Then $\sigma$ is modular, i.e., there exists a cuspidal non-lift Siegel modular eigenform $F$ of weight $k$, level one such that $L_{\rm spin}(s, F) = L(s, \sigma)$.  \end{thm}

\begin{proof}  Let $\sigma$ be as above. Then by Corollary \ref{latt}, there is a $G_{\Sigma}$-stable lattice $\Lambda$ in the space of $\sigma$ such that $$\ov{\sigma} = \bmat \chi^{k-2} & a & b \\ & \rho & c \\ && \chi^{k-1}\emat$$ with $a$ and $c$ non-trivial classes in $H^1_f(\bfQ, \rho(1-k))$.  Let $R^{\rm red}$ denote, as before, the (reduced, self-dual) universal deformation ring of $\ov{\sigma}$.

Let $\Phi'$ denote the set of linearly independent Siegel modular eigenforms of weight $k$, level $1$ which are orthogonal to the subspace spanned by Saito-Kurokawa lifts. By Theorems \ref{Gal rep for Sieg} and \ref{Jorza} we  get that for every $F \in \Phi'$ there exists a Galois representation $$\sigma_F : G_{\Sigma} \to \GL_4(E)$$ which is  short crystalline, satisfies $\det \sigma_F=\epsilon^{4k-6}$ and is $\tau$-self-dual.  Let $\Phi \subset \Phi'$ be the subset consisting of forms $F$ such that $\ov{\sigma}_F^{\rm ss} \cong \chi^{k-1} \oplus \rho \oplus \chi^{k-2}$.   For $F \in \Phi$ the representation $\sigma_F$ is absolutely irreducible which can be proven as in \cite{Brown07} (cf. the arguments following the proof of Proposition 8.3). Write $\bfT$ for the quotient of $\bfT'_{\fm}$ acting on the space spanned by $\Phi$. Here $\bfT'$ denotes the Hecke algebra generated over $\Oo$ by the operators $T_{\ell, 0}$, $T_{\ell, 1}$ and  $T_{\ell,2}$ for all primes $\ell \neq p$ and $\bfT'_{\fm}$ denotes the completion of $\bfT'$ at the maximal ideal corresponding to $\Phi$.

Thanks to Kato's result (Theorem 17.14 in \cite{Kato04})  we can use an argument similar to that in the proof of Proposition \ref{prop3.8} to deduce that \be \label{Kato} \#H^1_f(\bfQ, \rho_f(1-k)\otimes E/\Oo)\leq \# \Oo/L^{\rm alg}(k, f).\ee
  Here  $L^{\rm alg}(k, f)$ is the algebraic part of the  $L$-function of $f$ (for precise normalization see \cite{Brown11}). For this we apply an analogue
of the control theorem (Theorem \ref{control}) for $\xi=\epsilon^{-m}$ with $m=k-1$ and the fact that  the $p$-adic $L$-function interpolates the classical $L$-value at critical points. 

 We will now demonstrate that $\Phi$ is non-empty by showing that $\bfT\neq 0$. 
Since the Selmer group on the left of \eqref{Kato} is non-trivial, we get that  $\val_{\varpi}(L^{\rm alg}(k, f))>0$.  Let $\SK(f)$ denote the Saito-Kurokawa lift of $f$. It is a Siegel modular eigenform of weight $k$ and level $1$.
 Let $J$ be the ideal of $\bfT$ generated by the image of the annihilator ${\rm Ann}_{\bfT'_{\fm}}(\SK(f))$ of $\SK(f)$ under the map $\bfT'_{\fm} \twoheadrightarrow \bfT$. Then it follows from Corollary 5.6 in \cite{Brown11} that $\#\bfT/J \geq \# \Oo/L^{\rm alg}(k, f)$. In particular $\Phi \neq \emptyset$.

For every $F\in \Phi$, we get by Corollary \ref{latt} that we may assume that $\ov{\sigma}_F = \bmat \chi^{k-2} & *_1 & *_2 \\ & \rho & *_3 \\ && \chi^{k-1}\emat$. Since $\sigma_F$ are short crystalline (by our assumption that $p > 2k-2$) Proposition \ref{uni1} then shows that we have $\ov{\sigma}_F \cong \ov{\sigma}$. We further get an $\Oo$-algebra map $$\phi: R^{\rm red} \to \prod_{F \in \Phi} \Oo\supset \bfT.$$
Clearly $\phi(R^{\rm red}) \supset \bfT$, but we in fact get that $\phi(R^{\rm red})=\bfT$ by Corollary \ref{traces}. 
The map $\phi$ descends to a map $\ov{\phi}: R^{\rm red}/I^{\rm tot} \twoheadrightarrow \bfT/J$. 
Now combining Corollary \ref{structure map} with Theorem \ref{Lvalue bound}, inequality \eqref{Kato} and the fact $\#\bfT/J \geq \# \Oo/L^{\rm alg}(k, f)$, we conclude that $\ov{\phi}$ is an isomorphism. Since $I^{\rm tot}$ is principal by Theorem \ref{4}, $\phi$ is also an isomorphism by the commutative algebra criterion \cite{BergerKlosin13}, Theorem 4.1. 
  \end{proof}

\bibliographystyle{amsalpha}
\bibliography{standard2}


\appendix
\section{by  Cris Poor$^3$, Jerry Shurman$^4$, and David S. Yuen$^5$}

We prove the following theorem.

\begin{thm}\label{appendixtheorem}
The space $\StwoKsto$ has dimension~$19$, and its subspace of
Gritsenko lifts has dimension~$18$.  
There exists a new Hecke eigenform $f_{731}\in\StwoKsto$ such that:
\begin{enumerate}
\item $f_{731}$ has integer Fourier coefficients with content~$1$.
\item $f_{731}$ is not a Gritsenko lift.
\item $f_{731}$ is congruent modulo~$5$ to a Gritsenko lift Hecke
  eigenform that has integer Fourier coefficients.
\item the first few Hecke eigenvalues of $f_{731}$ are
$$
\lambda_2=-1,\ \lambda_3=0,\ \lambda_4=-2,\ \lambda_5=0,
$$
and its spin $2$-Euler factor is
$$
1+x+2x^2+2x^3+4x^4.
$$
\end{enumerate}
\end{thm}

The proof will take several steps.
The existence argument is motivated by techniques in
\cite{psyinflation,psysuper,MR3713095}, while the congruence is shown using a new 
technique.

\subsection{Notation and Jacobi Restriction method}

Let $\SkKN$ denote the space of weight~$k$ Siegel cusp forms for the
level~$N$ paramodular group, or, more briefly, the space of
weight~$k$, level~$N$ paramodular cusp forms.
We use the notation ``$[d]$'' for the image under projection onto the first $d$ Jacobi
coefficients, and ``$(d)$'' for the subspace where the first $d-1$
Jacobi coefficients vanish; because all Jacobi coefficients have index
divisible by~$N$, this condition is that the Jacobi coefficients of
index $N,2N,\dotsc,(d-1)N$ are~$0$.  Thus
\begin{alignat*}2
\SkKN[d]&=
&&\ \text{$\SkKN$ elements truncated to the first $d$ Jacobi coefficients},\\
\SkKN(d)&=
&&\ \text{$\SkKN$ elements with vanishing first $d-1$ Jacobi coefficients},
\end{alignat*}
and there is an exact sequence
$$
0\lra\SkKN(d+1)\lra\SkKN\lra\SkKN[d]\lra0.
$$
Analogous notation is understood to hold with either Fricke eigenspace
$\SkKN^\pm$ of~$\SkKN$, or any Atkin--Lehner eigenspace, and 
the exact sequence persists.

Let $\Jkmcusp$ denote the space of weight~$k$, index~$m$ Jacobi cusp forms.
Using the Fourier--Jacobi expansions of paramodular cusp forms, we view
two maps as containments for simplicity,
$$
\SkKN\subset\bigoplus_{j=1}^\infty\JkjNcusp\subset\C^\infty.
$$
The latter containment entails some chosen ordering  of the index sets
for Fourier expansions of Jacobi forms.
For any prime~$p$, let a subscript~$p$ denote the map from subsets
of~$\C^\infty$ to subsets of~$\Fp^\infty$ that reduces integral
elements modulo~$p$,
$$
V_p = (V\cap\Z^\infty)\mymod p,\quad V\subset\C^\infty.
$$
Thus
$$
\SkKN_p\subset\bigoplus_{j=1}^\infty(\JkjNcusp)_p\subset\Fp^\infty.
$$
The spaces 
$\SkKN$ and $\Jkmcusp$ have
bases in~$\Z^\infty$ by  results of  \cite{shim75} and \cite{ez85}, respectively.  
Thus $\dim_{\Fp}\SkKN_p=\dim_\C\SkKN$
and $\dim_{\Fp}(\Jkmcusp)_p=\dim_\C\Jkmcusp$.
Extending the notation from above,
\begin{alignat*}2
\SkKN[d]_p&=&&\ \text{reduction modulo $p$ of elements of
  $\SkKN[d]\cap\Z^\infty$},\\
\SkKN_p&=&&\ \text{reduction modulo $p$ of elements in
  $\SkKN\cap\Z^\infty$},\\
\SkKN_p(d)&=&&\ \text{elements of $\SkKN_p$ with first $d-1$
  Jacobi coefficients~$0$},
\end{alignat*}
and in consequence of an integral basis of~$\SkKN$,
\begin{align*}
\dim\SkKN[d]&=\dim\SkKN[d]_p,\\
\dim\SkKN(d)&\le\dim\SkKN_p(d).
\end{align*}
Again, analogous notation is understood to hold with either Fricke
eigenspace or any Atkin--Lehner eigenspace of~$\SkKN$, and the
dimension relations persist.

For either Fricke eigenvalue $\epsilon=\pm1$,
the Jacobi Restriction method \cite{bpy16,ipy13} runs with $d$ as a
parameter, returning a basis of a finite-dimensional complex vector
space $\JRMJdeps$ of truncated formal Fourier--Jacobi expansions such
that
$$
\SkKN^\epsilon[d]\subset\JRMJdeps\subset\bigoplus_{j=1}^d\JkjNcusp.
$$
An additional parameter, $\detmax$, bounds the Fourier coefficient
index determinants in the calculation, thus making the method an
algorithm; this parameter must be chosen so that each space $\JkjNcusp$
with~$j\le d$ is determined by the Fourier coefficients whose indices
satisfy the determinant bound.  For simplicity we suppress $\detmax$
from the notation $\JRMJdeps$, along with the weight~$k$ and the level~$N$.
We can also run the Jacobi Restriction method modulo a prime~$p$.
In this case, the algorithm returns a basis of a finite dimensional
vector space $\JRMJdpeps$ over~$\Fp$ such that
$$
\SkKN^\epsilon[d]_p\subset\JRMJdpeps\subset\bigoplus_{j=1}^d(\JkjNcusp)_p.
$$
The Jacobi Restriction method can also be run on any Atkin--Lehner
eigen\-space.  For a vector of signs describing an Atkin--Lehner
eigenspace, the Jacobi Restriction method returns a basis of a space
$\JRMJdsigns$ or $\JRMJdpsigns$ that contains
$\SkKN^{\text{signs}}[d]$ or $\SkKN^{\text{signs}}[d]_p$.

The methods of \cite{MR3713095} let us try to show that
$\StwoKN(d)$ or $\StwoKN_p(d)$ is~$0$ for a certain $d$.
The method we will use here requires us to span enough of~$\SfourKN$,
namely a subspace of codimension less than $\dim\JkNcusp2N$.  
A formula from \cite{ik} gives $\dim\SfourKlevel{731}=972$, and we
also have $\dim\JkNcusp2{731}=18$ from \cite{ez85}.  

\subsection{Spanning $\SfourKsto$}

For $\epsilon=(-1)^k$, let $\Grit:\JkNcusp kN\lra\SkKN^\epsilon$
be the linear map sending a Jacobi cusp form to its Gritsenko lift.  
Let $g_1,\dotsc,g_{18}$ be a basis of the space of Gritsenko lifts
in the Fricke plus space~$\StwoKsto^+$, 
and let $\tilde g_1,\dotsc,\tilde g_{138}$ be a basis of the space of
Gritsenko lifts in the Fricke plus space~$\SfourKsto^+$. 
The set
$$
S=\{T_\ell(g_ig_j),\ \tilde g_k: 1\le i\le j\le18,\ 1\le\ell\le4,\ 1\le k\le138\}
$$
turns out to span at least a $719$-dimensional space in~$\SfourKsto^+$,
as shown by its Fourier coefficients for indices out to determinant~$206$.
Establishing this required expanding the~$g_i$ initially out to
indices of determinant~$3296$.  Thus $\dim\SfourKsto^+\ge719$, and we
may select a basis of a $719$-dimensional subspace.  
For $\ell\nmid N$, the Hecke operator~$T_\ell$ employed here 
is the $T_{\ell,2}$ of the main article.  

Note that $731=17\cdot43$.  Let $\SfourKsto^{+,-}$ denote the subspace
of $\SfourKsto$ where the Atkin--Lehner operators $\AL_{17}$
and~$\AL_{43}$ have respective eigenvalues~$\pm1$, and similarly
for $\SfourKsto^{-,+}$.
Jacobi Restriction with parameters $d=6$, $p=5$, and~$\detmax=18275$
on these Atkin--Lehner subspaces returns spaces 
$\JRMJ_{6,5}^{+,-}$ and $\JRMJ_{6,5}^{-,+}$ such that
\begin{align*}
S_4(K(731))^{+,-}[6]_5\subset&\JRMJ_{6,5}^{+,-},
\quad\text{$\dim_{\Ffive}\JRMJ_{6,5}^{+,-}=125$},\\
S_4(K(731))^{-,+}[6]_5\subset&\JRMJ_{6,5}^{-,+},
\quad\text{$\dim_{\Ffive}\JRMJ_{6,5}^{-,+}=128$}.
\end{align*}
Even though $719+125+128=972$, it is not mathematically certain that
$719$, $125$, and~$128$ are the dimensions of the Fricke plus space
and the two Atkin--Lehner subspaces of the Fricke minus space in
$\SfourKsto$; however, in practice we are confident that they are, and
this guided our search for spanning elements.

We construct a Borcherds product $f_{4,731}\in\SfourKsto^-$.
The website \cite{yuen17} gives the details of the construction, but
here we note that the Borcherds product's leading theta block was
located among more than $354,000$ candidate theta blocks in
$\JkNcusp4{731}$ of $q$-vanishing order~$2$.  
Its symmetrizations
$$
h^{+,-},h^{-,+}=f_{4,731}\pm\AL_{17}(f_{4,731})
\in\SfourKsto^{+,-},\SfourKsto^{-,+}
$$
are both nonzero.  To generate many more  elements of these 
Atkin--Lehner subspaces, we use a technique called {\em bootstrapping\/}
\cite{psywt3sqfree}, created to overcome the obstacle that, computationally,
a Hecke operator returns shorter truncations of paramodular Fourier expansions 
than it receives.  This shortening allows only a few iterations of
Hecke operators, even starting with very long vectors that are
expensive to compute.  A brief summary of bootstrapping in the present
context is as follows.
\begin{enumerate}
\item We have a sufficiently long expansion of $h^{+,-}$ to identify
  it modulo~$5$ in $\JRMJ_{6,5}^{+,-}$.  The vector that we identify
  it as modulo~$5$ is longer.
\item Using the longer vector in $\JRMJ_{6,5}^{+,-}$, we have
  sufficiently many Fourier coefficients to apply~$T_2$, obtaining a
  shorter expansion of $T_2h^{+,-}$ modulo~$5$.  But this vector is
  sufficiently long to identify in $\JRMJ_{6,5}^{+,-}$, again giving
  us a long enough expansion to apply~$T_2$ again, and so on.
\end{enumerate}
This technique shows that $T_2^ih^{+,-}$ for $i=0,1,\dotsc,124$
are linearly independent modulo~$5$, and so
$\dim_{\Ffive}\SfourKsto^{+,-}[6]_5\ge125$.
Similarly, bootstrapping with $h^{-,+}$ gives
$\dim_{\Ffive}\SfourKsto^{-,+}[6]_5\ge124$.
Thus altogether we have shown that
$\dim_{\Ffive}\SfourKsto^{-}[6]_5\ge125+124=249$,
and we may select a basis modulo~$5$ of a $249$-dimensional subspace of $\SfourKsto^{-}_5$.  
So, because $719+249=968$, altogether we have reductions modulo~$5$ of
linearly independent elements that span a codimension~$4$ subspace
of~$\SfourKsto$.  
This is comfortably beyond the needed codimension $\dim\JkNcusp2{731}-1=17$.

\subsection{Determining $\StwoKsto^-$}

The $H_4(N,2,1)^-$ method of Proposition 4.6 in \cite{MR3713095}
uses a low-codimension subspace in $\SfourKN^-$ to show that
$\StwoKN_p^-(2)=0$.  Briefly, the idea is as follows.
Suppose that some~$h$ in $\StwoKsto_5^-(2)$ is nonzero.
Then the space $W = h\,\Grit(\JkNcusp2{731})_5 \subseteq S_4(K(731))_5^-$
has dimension~$18$.
But certain Fourier coefficients of every element~$w$ of~$W$
must be zero modulo~$5$, because such $w$ are a particular kind of product.
We make a $249\times r$ matrix~$M$ of coefficients from a basis of the
known $249$-dimensional subspace of $\SfourKsto^{-}_5$ where the
$r$ coefficients of every $w\in W$ must be zero, and we compute 
the rank of~$M$ modulo~$5$ to be~$249$.  Because  the 249-dimensional
subspace has codimension at most~$4$ in $\SfourKsto^{-}_5$, 
we conclude that $\dim W\le 4$, contradicting $\dim W=18$.
This proves that $\StwoKsto_5^-(2)=0$,
which in turn proves that $\StwoKsto^-(2)=0$.
Further, Jacobi Restriction with parameters $d=1$, $p=5$, and $\detmax=1462$
gives $\JRMJ_{1,5}^-=0$, and so $\StwoKsto^-[1]=0$, 
and the exact sequence gives $\StwoKsto^-=0$.

\subsection{Determining $\StwoKsto^+$}

The $H_4(N,6,1)^+$ method of Proposition 4.6 in \cite{MR3713095} 
shows that $\StwoKsto_5^+(6)=0$, using very similar reasoning.
This implies that $\StwoKsto^+(6)=0$.  
These critical vanishing results are ultimately what allow us to 
prove equalities and congruences in $\StwoKsto$.  
The exact sequence says that the projection map
$\StwoKsto^+\lra\StwoKsto^+[5]$ is an isomorphism.
Further, Jacobi Restriction with parameters $d=5$, $p=5$, and $\detmax=1462$
gives $\dim\JRMJ_{5,5}^+=19$, so $\dim\StwoKsto^+[5]\le 19$.
Thus $\dim\StwoKsto^+\le 19$, which is to say that there is at most
one dimension of nonlift eigenforms.

The existence of a nonlift,  given in section~\ref{sec16},  
will complete the proof that  $\dim\StwoKsto^+\  = 19$.

\subsection{Borcherds products}

Our current best method for finding a paramodular form that is not a
Gritsenko lift in the low weight~$2$ and the fairly high level~$731$ is to try
constructing Borcherds products (see \cite{psyinflation,psyallbp,MR3713095}).

However, the algorithm from \cite{psyallbp} shows that there are no
nonlift Borcherds products in $\StwoKsto$.
The method from \cite{psysuper,psyinflation}  is to go to twice the level and
search for a Borcherds product in the  minus Fricke
subspace~$\StwoKlevel{1462}^-$.  Indeed, we find two Borcherds products
$$
f_{2,1462,a},f_{2,1462,b}\in\StwoKlevel{1462}^-.
$$
See the website \cite{yuen17} for their constructions.

\subsection{Tracing down from $\StwoKlevel{1462}^-$ to $\StwoKlevel{731}^+$.}
\label{sec16}

\def\tracedown{{\rm TrDn}}

We prove the existence of a nonlift in~$\StwoKlevel{731}^+$ by using
the trace down map from \cite{MR3713095},
$$
\tracedown: \StwoKlevel{1462} \to \StwoKlevel{731}^+.
$$
Here the codomain is the Fricke plus space because
$\StwoKlevel{731}^-=0$, as shown above.
Tracing down the Borcherds products $f_{2,1462,a}$ and $f_{2,1462,b}$
requires many more Fourier coefficients than we can directly compute.
So we run Jacobi Restriction on $\StwoKlevel{1462}^-$ with parameters
$d=4$ and $\detmax=5848$, and get $\dim\JRMJ_4^-=2$.
After identifying the projections of the Borcherds products in $\JRMJ_4^-$, 
we can compute two Fourier coefficients of 
$\tracedown{(f_{2,1462,a})}$ and $\tracedown{(f_{2,1462,b})}$, 
but this is insufficient for our purposes.  
 To remedy this, 
we use the fact that the Borcherds products are Fricke minus forms 
to compute further Fourier coefficients for indices that are beyond 
the fourth Jacobi coefficient but still within the $\detmax$ bound.  
This permits the computation of~$46$ 
 Fourier coefficients of 
$\tracedown{(f_{2,1462,a})}$ and $\tracedown{(f_{2,1462,b})}$.
The computed coefficients of $\tracedown{(f_{2,1462,a})}$ show that
it is linearly independent of the Gritsenko lifts in $\StwoKlevel{731}^+$.
This proves that $\dim\StwoKlevel{731}^+=19$.

Because there is only one nonlift eigenform in $\StwoKsto^+$, it must
scale to have rational coefficients, and then rescale to a nonlift
eigenform $f_{731}$ that has integer Fourier coefficients with
content~$1$.  Jacobi Restriction with parameters $d=5$ and
$\detmax=1462$ returns $\dim\JRMJ_5^+=19$.
Hence $\StwoKsto^+[5]=\JRMJ_{5}^+$ and we have have enough
coefficients to compute eigenvalues on the Hecke operators through~$T_5$.
Especially, the nonlift eigenform $f_{731}$ has eigenvalues
$$
\lambda_2=-1,\ \lambda_3=0,\ \lambda_4=-2,\ \lambda_5=0.
$$

We refer to \cite{robertsschmidt06,robertsschmidt07} for the theory of
global paramodular newforms, and note here that the nonlift eigenform
$f_{731}$ must be a newform because by \cite{py15} the spaces 
$\StwoKlevel1$, $\StwoKlevel{17}$, and $\StwoKlevel{43}$ have no nonlifts.  

\subsection{Proving the mod 5 congruence.}

\newcommand\trunc[1]{#1[5]}      
\newcommand\red[1]{#1_5}      
\newcommand\redtrunc[1]{#1[5]_5}

First we compute the action of $T_2$ on the $18$-dimensional space of
Gritsenko lifts in $\StwoKsto^+$.  The characteristic polynomial of~$T_2$
on this space, factored into irreducibles over~$\Q$, is
\begin{align*}
&(x-4)(x-2)^2(x-1)\\
&\quad\cdot(x^6-17 x^5+112x^4-358 x^3+566 x^2-400 x+101)\\
&\quad\cdot(x^8-25 x^7+264 x^6-1530 x^5+5286x^4-11046 x^3+13465x^2-8612x+2174).
\end{align*}
Thus there is one Gritsenko lift eigenform with $\lambda_2=4$, which
we denote~$g$ and scale to have integer coefficients with content~$1$.
The first few eigenvalues of~$g$ can be computed,
$$
\lambda_2=4,\ \lambda_3=5,\ \lambda_4=8,\ \lambda_5=5.
$$
We now prove that the Fourier expansion of~$f_{731}$ is congruent to an
integer multiple of the Fourier expansion of~$g$ modulo~$5$.
For any $h\in\StwoKsto^+$ with integer coefficients, introduce
correspondingly
$$
\trunc h\in\StwoKsto^+[5],\quad
\red h\in\StwoKsto^+_5,\quad
\redtrunc h\in\StwoKsto^+[5]_5.
$$
For any $v\in\StwoKsto^+[5]_5$, there are enough computed Fourier
coefficients of $T_2v$ to determine $T_2v$ as an element of $\StwoKsto^+[5]_5$,
and so we can compute a matrix representative of the map
$$
\redtrunc{T_2}:\StwoKsto^+[5]_5\lra\StwoKsto^+[5]_5.
$$
Basic linear algebra on the matrix representative shows that the
$\redtrunc{T_2}$-eigenspace with eigenvalue~$4$,
$$
V=\{h\in\StwoKsto^+[5]_5:\redtrunc{T_2}h=4h\},
$$
is $2$-dimensional.
Although we did not have enough Fourier coefficients to compute the
matrix representative of $\redtrunc{T_4}$ on $\StwoKsto^+[5]_5$, we do
have enough to compute its matrix representative on the 
space~$V$, 
which has a smaller set of determining Fourier coefficients, 
and then basic linear algebra shows that the
$\redtrunc{T_4}$-eigenspace with eigenvalue~$3$ in~$V$ is $1$-dimensional.
Both $\redtrunc{f_{731}}$ and $\redtrunc g$ lie in~$V$,
and both have $\redtrunc{T_4}$-eigenvalue~$3$.
They cannot be linearly independent, because then they would span~$V$,
giving the contradiction that $\redtrunc{T_4}|_V=3\operatorname{Id}_V$.
Thus $\redtrunc{f_{731}}$ and~$\redtrunc g$ are linearly dependent,
and so there exist an integer $\beta\ne0\mymod5$ such that
$\redtrunc{f_{731}}+\beta\redtrunc g=0\mymod5$.  This implies that
$$
\red{(f_{731}+\beta g)}\in \StwoKsto^+_5(6).
$$
We earlier proved the critical result that  $\StwoKsto_5^+(6)=0$, and this implies 
$\red{(f_{731}+\beta g)}=0$.  This shows that $f_{731}$ is
congruent to a multiple of~$g$ modulo~$5$.
The multiple of $g$ is still a Gritsenko lift, so $f_{731}$ is
congruent modulo~$5$ to a Gritsenko lift.
This completes the proof of Theorem~\ref{appendixtheorem}.

\smallskip

We thank Ralf Schmidt for helpful discussions.

\bibliographystyleapp{plain}
\bibliographyapp{appendix731}

\end{document}